\documentclass{amsart}
\newtheorem{theorem}{Theorem}
\newtheorem{thm}{Theorem}
\newtheorem*{thmm}{Theorem}
\newtheorem*{VL}{Vaaler's Lemma}

\newtheorem{lemma}[theorem]{Lemma}

\newtheorem*{mydef}{Definition}
\newtheorem*{HI}{Hilbert’s inequality}

\newtheorem{prop}{Proposition}
\numberwithin{equation}{section}

\usepackage{xcolor}
\usepackage{soul}
\usepackage{setspace}
\usepackage{newtxmath}
\usepackage{mathrsfs}
\usepackage{accents}
\usepackage{bigints}
\usepackage{dsfont}

\usepackage{scalerel}
\begin{document}

\title[distribution and moments of the error term]{distribution and moments of the error term in the lattice point counting problem for $3$-dimensional Cygan-Kor{\'a}nyi balls}

\author{YOAV A. GATH}
\address{Centre for Mathematical Sciences, Wilberforce Road, Cambridge CB3 0WA, United Kingdom.}
\curraddr{}
\email{yg395@cam.ac.uk}
\thanks{}



\dedicatory{}
\begin{abstract}
We study fluctuations of the error term for the number of integer lattice points lying inside a $3$-dimensional Cygan-Kor{\'a}nyi ball of large radius. We prove that the error term, suitably normalized, has a limiting value distribution which is absolutely continuous, and we provide estimates for the decay rate of the corresponding density on the real line. In addition, we establish the existence of all moments for the normalized error term, and we prove that these are given by the moments of the corresponding density.
\end{abstract}
\maketitle
\tableofcontents
\section{Introduction, notation and statement of results}
\subsection{Introduction}
Given an integer $q\geq1$, let 
\begin{equation*}
N_{q}(x)=\big|\big\{\mathbf{z}\in\mathbb{Z}^{2q+1}:|\mathbf{z}|_{{\scriptscriptstyle Cyg}}\leq x\big\}\big|
\end{equation*}
be the counting function for the number of integer lattice points which lie inside a $(2q+1)$-dimensional Cygan-Kor{\'a}nyi ball of large radius $x>0$, where $|\,\,|_{{\scriptscriptstyle Cyg}}$ is the Cygan-Kor{\'a}nyi norm defined by
\begin{equation*}
|\mathbf{u}|_{{\scriptscriptstyle Cyg}}=\Big(\Big(u^{2}_{1}+\cdots+u^{2}_{2q}\Big)^{2}+u^{2}_{2q+1}\Big)^{1/4}\,.
\end{equation*}
The problem of estimating $N_{q}(x)$ arises naturally from the homogeneous structure imposed on the $q$-th Heisenberg group $\mathbb{H}_{q}$ (when realized over $\mathbb{R}^{2q+1}$), namely we have
\begin{equation*}
N_{q}(x)=\big|\mathbb{Z}^{2q+1}\cap\delta_{x}\mathcal{B}\big|\,,
\end{equation*}
where $\delta_{x}\mathbf{u}=(xu_{{\scriptscriptstyle 1}},\ldots, xu_{{\scriptscriptstyle 2q}},x^{2}u_{{\scriptscriptstyle 2q+1}})$ with $x>0$ are the Heisenberg dilations, and $\mathcal{B}=\big\{\mathbf{u}\in\mathbb{R}^{2q+1}:|\mathbf{u}|_{{\scriptscriptstyle Cyg}}\leq1\big\}$ is the unit ball with respect to the Cygan-Kor{\'a}nyi norm (see \cite{garg2015lattice, gath2019analogue} for more details). It is clear that $N_{q}(x)$ will grow for large $x$ like $\textit{vol}\big(\mathcal{B}\big)x^{2q+2}$, where $\textit{vol}(\cdot)$ is the Euclidean
volume, and we shall be interested in the error term resulting from this approximation.
\begin{mydef}
Let $q\geq1$ be an integer. For $x>0$, define
\begin{equation*}
\mathcal{E}_{q}(x)=N_{q}(x)-\textit{vol}\big(\mathcal{B}\big)x^{2q+2}\,,
\end{equation*}
and set $\kappa_{q}=\sup\big\{\alpha>0:\big|\mathcal{E}_{q}(x)\big|\ll x^{2q+2-\alpha}\big\}$.
\end{mydef}
\noindent
We shall refer to the lattice point counting problem for $(2q+1)$-dimensional Cygan-Kor{\'a}nyi balls as the problem of determining the value of $\kappa_{q}$. This problem was first considered by Garg, Nevo and Taylor \cite{garg2015lattice}, who established the lower bound $\kappa_{q}\geq2$ for all integers $q\geq1$. For $q=1$, which is the case we shall be concerned with in the present paper, this lower bound was proven by the author to be sharp, that is $\kappa_{1}=2$ (see \cite{gath2017best}, Theorem $1.1$, and note that a different normalization is used for the exponent of the error term). Thus, the lattice point counting problem for $3$-dimensional Cygan-Kor{\'a}nyi balls is settled.\\
The behavior of the error term $\mathcal{E}_{q}(x)$ in the lattice point counting problem for $(2q+1)$-dimensional Cygan-Kor{\'a}nyi balls with $q>1$ is of an entirely different nature compared to the case $q=1$. In the higher dimensional case $q\geq3$, the best result available to date is $|\mathcal{E}_{q}(x)|\ll x^{2q-2/3}$ which was proved by the author (\cite{gath2019analogue}, Theorem $1$), and we also have (\cite{gath2019analogue}, Theorem $3$) the $\Omega$-result $\mathcal{E}_{q}(x)=\Omega\big(x^{2q-1}\big(\log{x}\big)^{1/4}\big(\log{\log{x}}\big)^{1/8}\big)$. It follows that $\frac{8}{3}\leq\kappa_{q}\leq3$. In regards to what should be the conjectural value of $\kappa_{q}$ in the case of $q\geq3$, it is known (\cite{gath2019analogue}, Theorem $2$) that $\mathcal{E}_{q}(x)$ has order of magnitude $x^{2q-1}$ in mean-square, which leads to the conjecture that $\kappa_{q}=3$. As for the case $q=2$, the conjectural value of $\kappa_{2}$ should be the same as in the higher dimensional case, namely $\kappa_{2}=3$.\\
Our goal in the present paper is to investigate the nature in which $N_{1}(x)$ fluctuates around its expected value $\textit{vol}\big(\mathcal{B}\big)x^{4}$. In the higher dimensional case $q\geq3$, this has been carried out by the author, and we have the following result (\cite{gathdistribution}, Theorem $1$ and Theorem $3$). 
\begin{thmm}[\cite{gathdistribution}]
Let $q\geq3$ be an integer, and let $\widehat{\mathcal{E}}_{q}(x)=\mathcal{E}_{q}(x)/x^{2q-1}$ be the suitably normalized error term. Then there exists a probability density $\mathcal{P}_{q}(\alpha)$ such that, for any (piecewise)-continuous function $\mathcal{F}$ satisfying the growth condition $|\mathcal{F}(\alpha)|\ll\alpha^{2}$, we have
\begin{equation*}
\lim\limits_{X\to\infty}\frac{1}{X}\int\limits_{ X}^{2X}\mathcal{F}\big(\widehat{\mathcal{E}}_{q}(x)\big)\textit{d}x=\int\limits_{-\infty}^{\infty}\mathcal{F}(\alpha)\mathcal{P}_{q}(\alpha)\textit{d}\alpha\,.
\end{equation*}
The density $\mathcal{P}_{q}(\alpha)$ can be extended to the whole complex plane $\mathbb{C}$ as an entire function of $\alpha$, which satisfies for any non-negative integer $j\geq0$, and any $\alpha\in\mathbb{R}$, $|\alpha|$ sufficiently large in terms of $j$ and $q$, the decay estimate
\begin{equation*}
\begin{split}
\big|\mathcal{P}^{(j)}_{q}(\alpha)\big|\leq\exp{\Big(-|\alpha|^{4-\beta/\log\log{|\alpha|}}\Big)}\,,
\end{split}
\end{equation*}
where $\beta>0$ is an absolute constant.
\end{thmm}
\noindent
We are going to establish an analogue of the above result in the case of $q=1$, where the suitable normalization of the error term is given by $\widehat{\mathcal{E}}_{1}(x)=\mathcal{E}_{1}(x)/x^{2}$. We shall see once more a distinction between the case $q=1$ and the higher dimensional case $q\geq3$, this time with respect to the probability density $\mathcal{P}_{q}(\alpha)$.
\subsection{Statement of the main results}
\begin{thm}
Let $\widehat{\mathcal{E}}_{1}(x)=\mathcal{E}_{1}(x)/x^{2}$ be the suitably normalized error term. Then $\widehat{\mathcal{E}}_{1}(x)$ has a limiting value distribution in the sense that, there exists a probability density $\mathcal{P}_{1}(\alpha)$ such that, for any (piecewise)-continuous function $\mathcal{F}$ of polynomial growth we have
\begin{equation}\label{eq:1.1}
\lim\limits_{X\to\infty}\frac{1}{X}\int\limits_{X}^{2X}\mathcal{F}\big(\widehat{\mathcal{E}}_{1}(x)\big)\textit{d}x=\int\limits_{-\infty}^{\infty}\mathcal{F}(\alpha)\mathcal{P}_{1}(\alpha)\textit{d}\alpha\,.
\end{equation}
The density $\mathcal{P}_{1}(\alpha)$ satisfies for any non-negative integer $j\geq0$, and any $\alpha\in\mathbb{R}$, $|\alpha|$ sufficiently large in terms of $j$, the decay estimate
\begin{equation}\label{eq:1.2}
\begin{split}
\big|\mathcal{P}^{(j)}_{1}(\alpha)\big|\leq\exp{\bigg(-\frac{\pi}{2}|\alpha|\exp{\big(\rho|\alpha|\big)}\bigg)}\,,
\end{split}
\end{equation}
where $\rho>0$ is an absolute constant.
\end{thm}
\noindent
\textbf{Remark.} In the particular case where $\mathcal{F}(\alpha)=\alpha^{j}$ with $j\geq1$ an integer, we have
\begin{equation*}
\lim\limits_{X\to\infty}\frac{1}{X}\int\limits_{ X}^{2X}\widehat{\mathcal{E}}^{j}_{1}(x)\textit{d}x=\int\limits_{-\infty}^{\infty}\alpha^{j}\mathcal{P}_{1}(\alpha)\textit{d}\alpha\,,
\end{equation*}
thereby establishing the existence of all moments (see the remark following Proposition $4$ in $\S5$ regarding quantitative estimates).\\\\
\textbf{Remark.} Note that the decay estimates \eqref{eq:1.2} for the probability density in the case where $q=1$ are much stronger compared to the corresponding ones in the higher dimensional case $q\geq3$ as stated in the introduction. Also, note that whereas for $q\geq3$, the density $\mathcal{P}_{q}(\alpha)$ extends to the whole complex plane as an entire function of $\alpha$, and in particular is supported on all of the real line, Theorem $1$ above makes no such claim in the case of $q=1$.\\\\
Our next result gives a closed form expression for all the integral moments of the density $\mathcal{P}_{1}(\alpha)$.
\begin{thm}
Let $j\geq1$ be an integer. Then the $j$-th integral moment of $\mathcal{P}_{1}(\alpha)$ is given by
\begin{equation}\label{eq:1.3}
\int\limits_{-\infty}^{\infty}\alpha^{j}\mathcal{P}_{1}(\alpha)\textit{d}\alpha=\sum_{s=1}^{j}\underset{\,\,\ell_{1},\,\ldots\,,\ell_{s}\geq1}{\sum_{\ell_{1}+\cdots+\ell_{s}=j}}\,\,\frac{j!}{\ell_{1}!\cdots\ell_{s}!}\underset{\,\,\,m_{s}>\cdots>m_{1}}{\sum_{m_{1},\,\ldots\,,m_{s}=1}^{\infty}}\prod_{i=1}^{s}\Xi(m_{i},\ell_{i})\,,
\end{equation}
where the series on the RHS of \eqref{eq:1.4} converges absolutely. For integers $m,\ell\geq1$, the term $\Xi(m,\ell)$ is given by
\begin{equation}\label{eq:1.4}
\Xi(m,\ell)=(-1)^{\ell}\bigg(\frac{1}{\sqrt{2}\pi}\bigg)^{\ell}\frac{\mu^{2}(m)}{m^{\ell}}\sum_{\textit{e}_{1},\ldots,\textit{e}_{\ell}=\pm1}\underset{\textit{e}_{1}k_{1}+\cdots+\textit{e}_{\ell}k_{\ell}=0}{\sum_{k_{1},\ldots,k_{\ell}=1}^{\infty}}\prod_{i=1}^{\ell}\frac{r_{{\scriptscriptstyle2}}\big(mk^{2}_{i}\big)}{k_{i}^{2}}\,,
\end{equation}
where $\mu(\cdot)$ is the M{\"o}bius function, and $r_{{\scriptscriptstyle2}}(\cdot)$ is the counting function for the number of representation of an integer as the sum of two squares. For $\ell=1$ the sum in \eqref{eq:1.4} is void, so by definition $\Xi(m,1)=0$. In particular, it follows that $\int_{-\infty}^{\infty}\alpha\mathcal{P}_{1}(\alpha)\textit{d}\alpha=0$, and $\int_{-\infty}^{\infty}\alpha^{j}\mathcal{P}_{1}(\alpha)\textit{d}\alpha<0$ for $j\equiv1\,(2)$ greater then one. 
\end{thm}
\noindent
\textbf{Remark.} There is a further distinction between the case $q=1$, and the higher dimensional case $q\geq3$ in the following aspect. For $q=1$, $\mathcal{P}_{1}(\alpha)$ is the probability density corresponding to the random series $\sum_{m=1}^{\infty}\upphi_{{\scriptscriptstyle1,m}}\big(X_{m}\big)$, where the $X_{m}$ are independent random variables uniformly distributed on the segment $[0,1]$, and the $\upphi_{{\scriptscriptstyle1,m}}(t)$ are real-valued continuous functions, periodic of period $1$, given by (see $\S3$)
\begin{equation*}
\upphi_{{\scriptscriptstyle1,m}}(t)=-\frac{\sqrt{2}}{\pi}\frac{\mu^{2}(m)}{m}\sum_{k=1}^{\infty}\frac{r_{{\scriptscriptstyle2}}\big(mk^{2}\big)}{k^{2}}\cos{(2\pi kt)}\,.
\end{equation*}
This is in sharp contrast to the higher dimensional case $q\geq3$, where the corresponding functions $\upphi_{{\scriptscriptstyle q,m}}(t)$ (see \cite{gathdistribution}, $\S2$ Theorem $4$) are aperiodic. Also, the presence of the factor $1/m$ in $\upphi_{{\scriptscriptstyle1,m}}(t)$, as apposed to $1/m^{3/4}$ in $\upphi_{{\scriptscriptstyle q,m}}(t)$ for $q\geq3$, is the reason for why we obtain the much stronger decay estimates \eqref{eq:1.2} compared to the higher dimensional case $q\geq3$.\\\\
\textbf{Notation and conventions.} The following notation will occur repeatedly throughout this paper. We use the Vinogradov asymptotic notation $\ll$, as well as the Big-$O$ notation. For positive quantities $X,Y>0$ we write $X\asymp Y$, to mean that $X\ll Y\ll X$. In addition, we define
\begin{equation*}
\begin{split}
&\textit{(1)}\quad r_{{\scriptscriptstyle2}}(m)=\sum_{a^{2}+b^{2}=m}1\quad;\quad\text{where the representation runs over }a,b\in\mathbb{Z}\,.\\\\
&\textit{(2)}\quad\mu(m)=\left\{
        \begin{array}{ll}
            1 & ;\, m=1\\\\
            (-1)^{\ell} & ;\, \text{if }m=p_{{\scriptscriptstyle1}}\cdots p_{{\scriptscriptstyle\ell}},\,\text{with } p_{{\scriptscriptstyle1}},\ldots,p_{{\scriptscriptstyle\ell}}\,\text{distinct primes}\\\\
            0 & ;\,\text{otherwise}
        \end{array}
    \right.\\\\
&\textit{(3)}\quad\psi(t)=t-[t]-1/2\,\,,\,\,\text{where }[t]=\text{max}\{m\in\mathbb{Z}: m\leq t\}\,.
\end{split}
\end{equation*}
\section{A Vorono\"{i}-type series expansion for $\mathcal{E}_{1}(x)/x^{2}$}
\noindent
In this section, we develop a Vorono\"{i}-type series expansion for $\mathcal{E}_{1}(x)/x^{2}$. The main result we shall set out to prove is the following.
\begin{prop}
Let $X>0$ be large. Then for $X\leq x\leq2X$ we have
\begin{equation}\label{eq:2.1}
\mathcal{E}_{1}(x)/x^{2}=-\frac{\sqrt{2}}{\pi}\sum_{m\,\leq\,X^{2}}\frac{r_{{\scriptscriptstyle 2}}(m)}{m}\cos{\big(2\pi\sqrt{m}x\big)}-2x^{-2}T\big(x^{2}\big)+O_{\epsilon}\big(X^{-1+\epsilon}\big)\,,
\end{equation}
for any $\epsilon>0$, where for $Y>0$, $T(Y)$ is given by
\begin{equation*}
T(Y)=\sum_{0\,\leq\,m\,\leq\,Y}r_{{\scriptscriptstyle2}}(m)\psi\Big(\big(Y^{2}-m^{2}\big)^{1/2}\Big)\,.
\end{equation*}
\end{prop}
\noindent
\textbf{Remark.} It is not difficult to show that $\big|x^{-2}T\big(x^{2}\big)\big|\ll x^{-\theta}$ for some $\theta>0$, and so one may replace this term by this bound which simplifies \eqref{eq:2.1}. However, we have chosen to retain the term $x^{-2}T\big(x^{2}\big)$ as we are going to show later on (see $\S3.1$) that its average order is much smaller.\\\\  
The proof of Proposition $1$ will be given in $\S2.2$. We shall first need to establish several results regarding weighted integer lattice points in Euclidean circles.
\subsection{Weighted integer lattice points in Euclidean circles}
We begin this subsection by proving Lemma $1$ stated below, which will then be combined with Lemma $2$ to prove Proposition $1$.
\begin{lemma}
For $Y>0$, let
\begin{equation*}
R(Y)=\sum_{0\,\leq\,m\,\leq\,Y}r_{{\scriptscriptstyle2}}(m)\bigg(1-\frac{m}{Y}\bigg)^{1/2}-\frac{2\pi}{3}Y\,.
\end{equation*}
Then
\begin{equation}\label{eq:2.2}
R(Y)=-\frac{1}{2\pi}\sum_{m\,\leq\,M}\frac{r_{{\scriptscriptstyle 2}}(m)}{m}\cos{\big(2\pi\sqrt{mY}\,\big)}+O_{\epsilon}\big(Y^{-1/2+\epsilon}\big)\,,
\end{equation}
for any $\epsilon>0$, where $M$ as any real number which satisfies $M\asymp Y$.
\end{lemma}
\begin{proof}
Let $Y>0$ be large, $T\asymp Y$ a parameter to be specified later, and set $\delta=\frac{1}{\log{Y}}$. Write $\upphi$ for the continuous function on $\mathbb{R}_{>0}$ defined by $\upphi(t)=(1-t)^{1/2}$ if $t\in(0,1]$, and $\upphi(t)=0$ otherwise. We have
\begin{equation}\label{eq:2.3}
\begin{split}
&\sum_{0\,\leq\,m\,\leq\,Y}r_{{\scriptscriptstyle2}}(m)\bigg(1-\frac{m}{Y}\bigg)^{1/2}=1+\sum_{m=1}^{\infty}r_{{\scriptscriptstyle2}}(m)\upphi\bigg(\frac{m}{Y}\bigg)=1+\frac{1}{2\pi i}\int\limits_{1+\delta-i\infty}^{1+\delta+i\infty}Z(s)\check{\upphi}(s)Y^{s}\textit{d}s\\
&=1+\frac{1}{2\pi i}\int\limits_{1+\delta-iT}^{1+\delta+iT}Z(s)\check{\upphi}(s)Y^{s}\textit{d}s+O\bigg(YT^{-3/2}\sum_{m=1}^{\infty}\frac{r_{{\scriptscriptstyle2}}(m)}{m^{1+\delta}}\bigg(1+\text{min}\bigg\{T,\frac{1}{|\log{\frac{Y}{m}}|}\bigg\}\bigg)\bigg)\\
&=1+\frac{1}{2\pi i}\int\limits_{1+\delta-iY}^{1+\delta+iY}Z(s)\check{\upphi}(s)Y^{s}\textit{d}s+O_{\epsilon}\big(Y^{-1/2+\epsilon}\big)\,,
\end{split}
\end{equation}
where $\check{\upphi}(s)=\frac{\Gamma(s)\Gamma(3/2)}{\Gamma(s+3/2)}$ is the Mellin transform of $\upphi$, and $Z(s)=\sum_{m=1}^{\infty}r_{{\scriptscriptstyle2}}(m)m^{-s}$ with $\Re(s)>1$. In estimating \eqref{eq:2.3}, we have made use of Stirling's asymptotic formula for the Gamma function (see \cite{iwaniec2004analytic}, A.4 (5.113))
\begin{equation}\label{eq:2.4}
\Gamma(\sigma+it)=\sqrt{2\pi}(it)^{\sigma-\frac{1}{2}}e^{-\frac{\pi}{2}|t|}\bigg(\frac{|t|}{e}\bigg)^{it}\bigg(1+O\bigg(\frac{1}{|t|}\bigg)\bigg)\,,
\end{equation}
valid uniformly for $\alpha<\sigma<\beta$ with any fixed $\alpha,\beta\in\mathbb{R}$, provided $|t|$ is large enough in terms of $\alpha$ and $\beta$.\\
The Zeta function $Z(s)$, initially defined for $\Re(s)>1$, admits an analytic continuation to the entire complex plane, except at $s=1$ where it has a simple pole with residue $\pi$, and satisfies the functional equation \cite{epstein1903}
\begin{equation}\label{eq:2.5}
\pi^{-s}\Gamma(s)Z(s)=\pi^{-(1-s)}\Gamma(1-s)Z(1-s)\,.
\end{equation}
Now, $s(s-1)Z(s)\check{\upphi}(s)Y^{s}$ is regular in the strip $-\delta\leq\Re(s)\leq 1+\delta$, and by Stirling's asymptotic formula \eqref{eq:2.4} together with the functional equation \eqref{eq:2.5}, we obtain the bounds
\begin{equation}\label{eq:2.6}
\begin{split}
&\big|s(s-1)Z(s)\check{\upphi}(s)Y^{s}\big|\ll Y(\log{Y})\big(1+|s|\big)^{1/2}\quad\,\,\,\,;\quad\Re(s)=1+\delta\\
&\big|s(s-1)Z(s)\check{\upphi}(s)Y^{s}\big|\ll (\log{Y})\big(1+|s|\big)^{3/2+2\delta}\quad;\quad\Re(s)=-\delta\,.
\end{split}
\end{equation}
On recalling that $T\asymp Y$, it follows from the Phragm{\'e}n-{L}indel{\"o}f principle that
\begin{equation}\label{eq:2.7}
\begin{split}
\big|Z(s)\check{\upphi}(s)Y^{s}\big|&\ll(\log{Y}) T^{\delta-1/2}\bigg(\frac{Y}{T}\bigg)^{\sigma}\\
&\ll Y^{-1/2}\log{Y}\quad;\qquad-\delta\leq\sigma=\Re(s)\leq1+\delta\,\,,\,\,|\Im(s)|=T\,.
\end{split}
\end{equation}
Moving the line of integration to $\Re(s)=-\delta$, and using \eqref{eq:2.7}, we have by the theorem of residues
\begin{equation}\label{eq:2.8}
\begin{split}
\frac{1}{2\pi i}\int\limits_{1+\delta-iY}^{1+\delta+iY}Z(s)\check{\upphi}(s)Y^{s}\textit{d}s&=\bigg\{\underset{s=1}{\text{Res}}+\underset{s=0}{\text{Res}}\bigg\}Z(s)\check{\upphi}(s)Y^{s}+\frac{1}{2\pi i}\int\limits_{-\delta-iY}^{-\delta+iY}Z(s)\check{\upphi}(s)Y^{s}\textit{d}s\\
&\,\,\,\,\,\,+O\big(Y^{-1/2}\log{Y}\big)\\
&=\frac{2\pi}{3}Y-1+\frac{1}{2\pi i}\int\limits_{-\delta-iY}^{-\delta+iY}Z(s)\check{\upphi}(s)Y^{s}\textit{d}s+O\big(Y^{-1/2}\log{Y}\big)\,.
\end{split}
\end{equation}
Inserting \eqref{eq:2.8} into the the RHS of \eqref{eq:2.3}, and applying the functional equation \eqref{eq:2.5}, we arrive at
\begin{equation}\label{eq:2.9}
R(Y)=\frac{1}{2\sqrt{\pi}}\sum_{m=1}^{\infty}\frac{r_{{\scriptscriptstyle2}}(m)}{m}\text{\Large{J}}_{m}+O_{\epsilon}\big(Y^{-1/2+\epsilon}\big)\,,
\end{equation}
where
\begin{equation}\label{eq:2.10}
\text{\Large{J}}_{m}=\frac{1}{2\pi i}\int\limits_{-\delta-iT}^{-\delta+iT}\frac{\Gamma(1-s)\big(\pi^{2}mY\big)^{s}}{\Gamma(s+3/2)}\textit{d}s\,.
\end{equation}
Now, let $M>0$ be a real number which satisfies $M\asymp Y$. We then specify the parameter $T$ by making the choice $T=\pi\sqrt{MY}$. Clearly, we have $T\asymp Y$. We are going to estimate $\text{\Large{J}}_{m}$ separately for $m\leq M$ and $m>M$.\\
Suppose first that $m>M$. We have 
\begin{equation}\label{eq:2.11}
\begin{split}
\bigg|\frac{\Gamma(1-s)\big(\pi^{2}mY\big)^{s}}{\Gamma(s+3/2)}\bigg|&\ll T^{-1/2}\big(mT^{-2}\pi^{2}Y\big)^{\sigma}\\
&=T^{-1/2}\bigg(\frac{m}{M}\bigg)^{\sigma}\\
&\ll Y^{-1/2}m^{-\delta}\quad;\qquad-\frac{1}{4}\leq\sigma=\Re(s)\leq-\delta\,\,,\,\,|\Im(s)|=T\,.
\end{split}
\end{equation}
Moving the line of integration to $\Re(s)=-\frac{1}{4}$, and using \eqref{eq:2.11}, we obtain
\begin{equation}\label{eq:2.12}
\text{\Large{J}}_{m}=\frac{1}{2\pi i}\int\limits_{-\frac{1}{4}-iT}^{-\frac{1}{4}+iT}\frac{\Gamma(1-s)\big(\pi^{2}mY\big)^{s}}{\Gamma(s+3/2)}\textit{d}s+O\big(Y^{-1/2}m^{-\delta}\big)\,.
\end{equation}
By Stirling's asymptotic formula \eqref{eq:2.4} it follows that
\begin{equation}\label{eq:2.13}
\big|\text{\Large{J}}_{m}\big|\ll\big(mY\big)^{-1/4}\bigg\{1+\bigg|\int\limits_{1}^{T}\exp{\big(if_{{\scriptscriptstyle m}}(t)\big)}\textit{d}t\bigg|\,\bigg\}+Y^{-1/2}m^{-\delta}\,,
\end{equation}
where $f_{{\scriptscriptstyle m}}(t)=-2t\log{t}+2t +t\log{\big(\pi^{2}mY\big)}$. Trivial integration and integration by parts give
\begin{equation}\label{eq:2.14}
\bigg|\int\limits_{1}^{T}\exp{\big(if_{{\scriptscriptstyle m}}(t)\big)}\textit{d}t\bigg|\ll\text{min}\bigg\{T,\frac{1}{\log{\frac{m}{M}}}\bigg\}\,.
\end{equation}
Inserting \eqref{eq:2.14} into \eqref{eq:2.13}, and then summing over all $m>M$, we obtain
\begin{equation}\label{eq:2.15}
\begin{split}
\sum_{m>M}\frac{r_{{\scriptscriptstyle2}}(m)}{m}\big|\text{\Large{J}}_{m}\big|&\ll Y^{-1/4}\sum_{m>M}\frac{r_{{\scriptscriptstyle2}}(m)}{m^{5/4}}\bigg(1+\text{min}\bigg\{T,\frac{1}{\log{\frac{m}{M}}}\bigg\}\bigg)+Y^{-1/2}\log{Y}\\
&\ll_{\epsilon} Y^{-1/2+\epsilon}\,.
\end{split}
\end{equation}
Inserting \eqref{eq:2.15} into \eqref{eq:2.9}, we arrive at
\begin{equation}\label{eq:2.16}
R(Y)=\frac{1}{2\sqrt{\pi}}\sum_{m\leq M}\frac{r_{{\scriptscriptstyle2}}(m)}{m}\text{\Large{J}}_{m}+O_{\epsilon}\big(Y^{-1/2+\epsilon}\big)\,.
\end{equation}
It remains to estimate $\text{\Large{J}}_{m}$ for $m\leq M$. We have 
\begin{equation}\label{eq:2.17}
\begin{split}
\bigg|\frac{\Gamma(1-s)\big(\pi^{2}mY\big)^{s}}{\Gamma(s+3/2)}\bigg|&\ll T^{-1/2}\big(mT^{-2}\pi^{2}Y\big)^{\sigma}\\
&=T^{-1/2}\bigg(\frac{m}{M}\bigg)^{\sigma}\\
&\ll Y^{-1/2}\quad;\qquad-\delta\leq\sigma=\Re(s)\leq1-\delta\,\,,\,\,|\Im(s)|=T\,.
\end{split}
\end{equation}
Moving the line of integration to $\Re(s)=1-\delta$, and using \eqref{eq:2.17}, we obtain
\begin{equation}\label{eq:2.18}
\text{\Large{J}}_{m}=\frac{1}{2\pi i}\int\limits_{1-\delta-iT}^{1-\delta+iT}\frac{\Gamma(1-s)\big(\pi^{2}mY\big)^{s}}{\Gamma(s+3/2)}\textit{d}s+O\big(Y^{-1/2}\big)\,.
\end{equation}
Extending the integral all the way to $\pm\infty$, by Stirling's asymptotic formula \eqref{eq:2.4} we have
\begin{equation}\label{eq:2.19}
\begin{split}
\frac{1}{2\pi i}&\int\limits_{1-\delta-iT}^{1-\delta+iT}\frac{\Gamma(1-s)\big(\pi^{2}mY\big)^{s}}{\Gamma(s+3/2)}\textit{d}s\\
&=\frac{1}{2\pi i}\int\limits_{1-\delta-i\infty}^{1-\delta+i\infty}\frac{\Gamma(1-s)\big(\pi^{2}mY\big)^{s}}{\Gamma(s+3/2)}\textit{d}s+O\bigg(Y^{2}\bigg|\int\limits_{T}^{\infty}t^{-5/2+2\delta}\exp{\big(if_{{\scriptscriptstyle m}}(t)\big)}\textit{d}t\bigg|+Y^{-1/2}\bigg)\,,
\end{split}
\end{equation}
where $f_{{\scriptscriptstyle m}}(t)$ is defined as before. Trivial integration and integration by parts give
\begin{equation}\label{eq:2.20}
\bigg|\int\limits_{T}^{\infty}t^{-5/2+2\delta}\exp{\big(if_{{\scriptscriptstyle m}}(t)\big)}\textit{d}t\bigg|\ll T^{-5/2}\text{min}\bigg\{T,\frac{1}{\log{\frac{M}{m}}}\bigg\}\,.
\end{equation}
Inserting \eqref{eq:2.19} into the RHS of \eqref{eq:2.18}, we obtain by \eqref{eq:2.20}
\begin{equation}\label{eq:2.21}
\text{\Large{J}}_{m}=\frac{1}{2\pi i}\int\limits_{1-\delta-i\infty}^{1-\delta+i\infty}\frac{\Gamma(1-s)\big(\pi^{2}mY\big)^{s}}{\Gamma(s+3/2)}\textit{d}s+O\bigg(Y^{-1/2}\bigg(1+\text{min}\bigg\{T,\frac{1}{\log{\frac{m}{M}}}\bigg\}\bigg)\bigg)
\end{equation}
Moving the line of integration in \eqref{eq:2.21} to $\Re(s)=N+1/2$ with $N\geq1$ an integer, and then letting $N\to\infty$, we have by the theorem of residues
\begin{equation}\label{eq:2.22}
\frac{1}{2\pi i}\int\limits_{1-\delta-i\infty}^{1-\delta+i\infty}\frac{\Gamma(1-s)\big(\pi^{2}mY\big)^{s}}{\Gamma(s+3/2)}\textit{d}s=\big(\pi\sqrt{mY}\,\big)^{1/2}\mathcal{J}_{{\scriptscriptstyle3/2}}\big(2\pi\sqrt{mY}\,\big)\,,
\end{equation}
where for $\nu>0$, the Bessel function $\mathcal{J}_{{\scriptscriptstyle\nu}}$ of order $\nu$ is defined by
\begin{equation*}
\mathcal{J}_{{\scriptscriptstyle3/2}}(y)=\sum_{k=0}^{\infty}\frac{(-1)^{k}}{k!\Gamma(k+1+\nu)}\bigg(\frac{y}{2}\bigg)^{\nu+2k}\,.
\end{equation*}
We have the following asymptotic estimate for the Bessel function (see \cite{iwaniec2002spectral}, B.4 (B.35)). For fixed $\nu>0$,
\begin{equation}\label{eq:2.23}
\mathcal{J}_{\nu}(y)=\bigg(\frac{2}{\pi y}\bigg)^{1/2}\cos{\bigg(y-\frac{1}{2}\nu\pi-\frac{1}{4}\pi\bigg)}+O\bigg(\frac{1}{y^{3/2}}\bigg)\,,\text{ as } y\to\infty\,.
\end{equation}
Inserting \eqref{eq:2.22} into the RHS of \eqref{eq:2.21}, we have by \eqref{eq:2.23}
\begin{equation}\label{eq:2.24}
\text{\Large{J}}_{m}=-\frac{1}{\sqrt{\pi}}\cos{\big(2\pi\sqrt{mY}\,\big)}+O\bigg(Y^{-1/2}\bigg(1+\text{min}\bigg\{T,\frac{1}{\log{\frac{m}{M}}}\bigg\}\bigg)\bigg)\,.
\end{equation}
Summing over all $m\leq M$, we obtain
\begin{equation}\label{eq:2.25}
\begin{split}
\sum_{m\leq M}&\frac{r_{{\scriptscriptstyle2}}(m)}{m}\text{\Large{J}}_{m}\\
&=-\frac{1}{\sqrt{\pi}}\sum_{m\leq M}\frac{r_{{\scriptscriptstyle2}}(m)}{m}\cos{\big(2\pi\sqrt{mY}\,\big)}+O\bigg(Y^{-1/2}\sum_{m\leq M}\frac{r_{{\scriptscriptstyle2}}(m)}{m}\bigg(1+\text{min}\bigg\{T,\frac{1}{\log{\frac{m}{M}}}\bigg\}\bigg)\bigg)\\
&=-\frac{1}{\sqrt{\pi}}\sum_{m\leq M}\frac{r_{{\scriptscriptstyle2}}(m)}{m}\cos{\big(2\pi\sqrt{mY}\,\big)}+O_{\epsilon}\big(Y^{-1/2+\epsilon}\big)\,.
\end{split}
\end{equation}
Finally, inserting \eqref{eq:2.25} into the RHS of \eqref{eq:2.16}, we arrive at
\begin{equation}\label{eq:2.26}
R(Y)=-\frac{1}{2\pi}\sum_{m\,\leq\,M}\frac{r_{{\scriptscriptstyle 2}}(m)}{m}\cos{\big(2\pi\sqrt{mY}\,\big)}+O_{\epsilon}\big(Y^{-1/2+\epsilon}\big)\,.
\end{equation}
This concludes the proof.
\end{proof}
\noindent
We need an additional result regarding weighted integer lattice points in Euclidean circles. First, we make the following definition.
\begin{mydef}
For $Y>0$ and $n\geq1$ an integer, define
\begin{equation*}
\begin{split}
S_{{\scriptscriptstyle n}}(Y)=\frac{(-1)^{n}}{n!}f^{(n)}\big(2Y\big)\sum_{0\,\leq\,m\,\leq\,Y}r_{{\scriptscriptstyle 2}}(m)\big(Y-m\big)^{n+1/2}\,,
\end{split}
\end{equation*}
with $f(y)=\sqrt{y}$, and let
\begin{equation*}
R_{{\scriptscriptstyle n}}(Y)=S_{{\scriptscriptstyle n}}(Y)-c_{{\scriptscriptstyle n}}Y^{2}\,,
\end{equation*}
where $c_{{\scriptscriptstyle0}}=\frac{2^{3/2}\pi}{3}$, $c_{{\scriptscriptstyle1}}=-\frac{\pi}{5\sqrt{2}}$, and $c_{{\scriptscriptstyle n}}=-\pi\frac{\prod_{k=1}^{n-1}\big(1-\frac{1}{2k}\big)}{\sqrt[]{2}\,2^{n}n\big(n+\frac{3}{2}\big)}$ for $n\geq2$.
\end{mydef}
\noindent
We quote the following result (see \cite{gath2017best}, Lemma $2.1$). Here, we shall only need that part of the lemma which concerns the case $n\geq1$. The case $n=0$ will be treated by Lemma $1$ as we shall see later.
\begin{lemma}[\cite{gath2017best}]
For $n\geq1$ an integer, the error term $R_{{\scriptscriptstyle n}}(Y)$ satisfies the bound
\begin{equation}\label{eq:2.27}
\big|R_{{\scriptscriptstyle n}}(Y)\big|\ll2^{-n}Y^{1/2}\,,
\end{equation}
where the implied constant is absolute.
\end{lemma}
\noindent
\textbf{Remark.} The proof of Lemma $2$ goes along the same line as the proof of Lemma $1$, where in fact the proof is much simpler in this case. Moreover, one can show that the upper bound estimate \eqref{eq:2.27} is sharp for $n=1$. For $n\geq2$, the estimates are no longer sharp, but they will more then suffice for our needs. 
\subsection{A decomposition identity for $N_{1}(x)$ and proof of Proposition $1$} We have everything we need for the proof of Proposition $1$. Before presenting the proof, we need the following decomposition identity for $N_{1}(x)$ which we prove in Lemma $3$ below.
\begin{lemma}
Let $x>0$. We have
\begin{equation}\label{eq:2.28}
\begin{split}
N_{1}(x)&=2\sum_{n=0}^{\infty}S_{{\scriptscriptstyle n}}\big(x^{2}\big)-2T\big(x^{2}\big)\,,
\end{split}
\end{equation}
where $T\big(x^{2}\big)$ is defined as in Proposition $1$. 
\end{lemma}
\begin{proof}
The first step is to execute the lattice point count as follows. By the definition of the Cygan-Kor{\'a}nyi norm, we have 
\begin{equation}\label{eq:2.29}
\begin{split}
N_{1}(x)&=\sum_{m^{2}+n^{2}\leq\,x^{4}}r_{{\scriptscriptstyle2}}(m)\\
&=2\sum_{0\,\leq\,m\,\leq\,x^{2}}r_{{\scriptscriptstyle 2}}(m)\big(x^{4}-m^{2}\big)^{1/2}-2\sum_{0\,\leq\,m\,\leq\,x^{2}}r_{{\scriptscriptstyle2}}(m)\psi\Big(\big(x^{4}-m^{2}\big)^{1/2}\Big)\\
&=2\sum_{0\,\leq\,m\,\leq\,x^{2}}r_{{\scriptscriptstyle 2}}(m)\big(x^{4}-m^{2}\big)^{1/2}-2T\big(x^{2}\big)\,.
\end{split}
\end{equation}
Next, we decompose the first sum using the following procedure. For $0\leq m\leq x^{2}$ an integer, we use Taylor expansion to write
\begin{eqnarray*}
\big(x^{4}-m^{2}\big)^{1/2}=\sum_{n=0}^{\infty}\frac{(-1)^{n}}{n!}f^{(n)}\big(2x^{2}\big)\big(x^{2}-m\big)^{n+1/2}\,,
\end{eqnarray*}
with $f(y)=\sqrt{y}$. Multiplying the above identity by $r_{{\scriptscriptstyle 2}}(m)$, and then summing over the range $0\leq m\leq x^{2}$, we have
\begin{equation}\label{eq:2.30}
\begin{split}
\sum_{0\,\leq\,m\,\leq\,x^{2}}r_{{\scriptscriptstyle 2}}(m)\big(x^{4}-m^{2}\big)^{1/2}&=\sum_{n=0}^{\infty}\frac{(-1)^{n}}{n!}f^{(n)}\big(2x^{2}\big)\sum_{0\,\leq\,m\,\leq\,x^{2}}r_{{\scriptscriptstyle 2}}(m)\big(x^{2}-m\big)^{n+1/2}\\
&=\sum_{n=0}^{\infty}S_{{\scriptscriptstyle n}}\big(x^{2}\big)\,.
\end{split}
\end{equation}
Inserting \eqref{eq:2.30} into the RHS of \eqref{eq:2.29}, we obtain
\begin{equation}\label{eq:2.31}
\begin{split}
N_{1}(x)&=2\sum_{n=0}^{\infty}S_{{\scriptscriptstyle n}}\big(x^{2}\big)-2T\big(x^{2}\big)\,.
\end{split}
\end{equation}
This concludes the proof.
\end{proof}
\noindent
We now turn to the proof of Proposition $1$.
\begin{proof}(Proposition 1). Let $X>0$ be large, and suppose that $X<x<2X$. By the upper bound estimate \eqref{eq:2.27} in Lemma $2$, it follows from Lemma $3$ that
\begin{equation}\label{eq:2.32}
\begin{split}
N_{1}(x)&=2\sum_{n=0}^{\infty}S_{{\scriptscriptstyle n}}\big(x^{2}\big)-2T\big(x^{2}\big)\\
&=2\sum_{n=0}^{\infty}\Big\{c_{{\scriptscriptstyle n}}x^{4}+R_{{\scriptscriptstyle n}}\big(x^{2}\big)\Big\}-2T\big(x^{2}\big)\\
&=2cx^{4}+2R_{{\scriptscriptstyle0}}\big(x^{2}\big)+2\sum_{n=1}^{\infty}R_{{\scriptscriptstyle n}}\big(x^{2}\big)-2T\big(x^{2}\big)\\
&=2cx^{4}+2R_{{\scriptscriptstyle0}}\big(x^{2}\big)-2T\big(x^{2}\big)+O(x)\,,
\end{split}
\end{equation}
with $c=\sum_{n=0}^{\infty}c_{{\scriptscriptstyle n}}$, where the infinite sum clearly converges absolutely. By the definition of $R_{{\scriptscriptstyle0}}(Y)$, it is easily verified that
\begin{equation}\label{eq:2.33}
\begin{split}
R_{{\scriptscriptstyle0}}(Y)&=2^{1/2}Y\Bigg\{\,\sum_{0\,\leq\,m\,\leq\,Y}r_{{\scriptscriptstyle 2}}(m)\bigg(1-\frac{m}{Y}\bigg)^{1/2}-\frac{2\pi}{3}Y\Bigg\}\\
&=2^{1/2}YR(Y)\,,
\end{split}
\end{equation}
Inserting \eqref{eq:2.33} into the RHS of \eqref{eq:2.32}, we find that
\begin{equation}\label{eq:2.34}
N_{1}(x)=2cx^{4}+2^{3/2}x^{2}R\big(x^{2}\big)-2T\big(x^{2}\big)+O(x)\,.
\end{equation}
We claim that $2c=\textit{vol}\big(\mathcal{B}\big)$. To see this, first note that by \eqref{eq:2.2} in Lemma $1$ we have the bound $|R\big(x^{2}\big)|\ll\log{x}$, and since $|T\big(x^{2}\big)|\ll x^{2}$, we obtain by \eqref{eq:2.34} that $N_{1}(x)=2cx^{4}+O(x^{2}\log{x})$. Since $N_{1}(x)\sim\textit{vol}\big(\mathcal{B}\big)x^{4}$ as $x\to\infty$, we conclude that $2c=\textit{vol}\big(\mathcal{B}\big)$.\\
Subtracting $\textit{vol}\big(\mathcal{B}\big)x^{4}$ from both sides of \eqref{eq:2.34}, and then dividing throughout by $x^{2}$, we have by \eqref{eq:2.2} in Lemma $1$ upon choosing $M=X^{2}$
\begin{equation}\label{eq:2.35}
\mathcal{E}_{1}(x)/x^{2}=-\frac{\sqrt{2}}{\pi}\sum_{m\,\leq\,X^{2}}\frac{r_{{\scriptscriptstyle 2}}(m)}{m}\cos{\big(2\pi\sqrt{m}x\big)}-2x^{-2}T\big(x^{2}\big)+O_{\epsilon}\big(X^{-1+\epsilon}\big)\,.
\end{equation}
This concludes the proof of Proposition $1$.
\end{proof}
\section{Almost periodicity}
\noindent
In this section we show that the suitably normalized error term $\mathcal{E}_{1}(x)/x^{2}$ can be approximated, in a suitable sense, by means of certain oscillating series. From this point onward, we shall use the notation $\widehat{\mathcal{E}}_{1}(x)=\mathcal{E}_{1}(x)/x^{2}$. The main result we shall set out to prove is the following.
\begin{prop}
We have
\begin{equation}\label{eq:3.1}
\lim_{M\to\infty}\limsup_{X\to\infty}\frac{1}{X}\int\limits_{X}^{2X}\Big|\widehat{\mathcal{E}}_{1}(x)-\sum_{m\,\leq\,M}\upphi_{{\scriptscriptstyle 1,m}}\big(\sqrt{m}x\big)\Big|\textit{d}x=0\,,
\end{equation}
where $\upphi_{{\scriptscriptstyle 1,1}}(t), \upphi_{{\scriptscriptstyle 1,2}}(t),\ldots$ are real-valued continuous functions, periodic of period $1$, given by
\begin{equation*}
\upphi_{{\scriptscriptstyle1,m}}(t)=-\frac{\sqrt{2}}{\pi}\frac{\mu^{2}(m)}{m}\sum_{k=1}^{\infty}\frac{r_{{\scriptscriptstyle2}}\big(mk^{2}\big)}{k^{2}}\cos{(2\pi kt)}\,.
\end{equation*}
\end{prop}
\noindent
The proof of Proposition $2$ will be given in $\S3.2$. Our first task will be to deal with the remainder term $x^{-2}T\big(x^{2}\big)$ appearing in the approximate expression \eqref{eq:2.1}.
\subsection{Bounding the remainder term}
This subsection is devoted to proving the following lemma.
\begin{lemma}
We have
\begin{equation}\label{eq:3.2}
\frac{1}{X}\int\limits_{X}^{2X}T^{2}\big(x^{2}\big)\textit{d}x\ll X^{2}\big(\log{X}\big)^{4}\,.
\end{equation}
\end{lemma}
\noindent
Before commencing with the proof, we need the following result on trigonometric approximation for the $\psi$ function (see \cite{vaaler1985some}).
\begin{VL}[\cite{vaaler1985some}]
Let $H\geq1$. Then there exist trigonometrical polynomials
\begin{equation*}
\begin{split}
&\textit{(1)}\qquad\psi_{{\scriptscriptstyle H}}(\omega)=\sum_{1\,\leq\,h\,\leq\,H}\nu(h)\sin{(2\pi h\omega)}\\
&\textit{(2)}\qquad\psi^{\ast}_{{\scriptscriptstyle H}}(\omega)=\sum_{1\,\leq\,h\,\leq\,H}\nu^{\ast}(h)\cos{(2\pi h\omega)}\,,
\end{split}
\end{equation*}
with real coefficients satisfying $|\nu(h)|,\,|\nu^{\ast}(h)|\ll 1/h$, such that
\begin{equation*}
\big|\psi(\omega)-\psi_{{\scriptscriptstyle H}}(\omega)\big|\leq\psi^{\ast}_{{\scriptscriptstyle H}}(\omega)+\frac{1}{2[H]+2}\,.
\end{equation*}
\end{VL}
\noindent 
We now turn to the proof of Lemma $4$.
\begin{proof}(Lemma 4).
Let $X>0$ be large. By Vaaler's Lemma with $H=X$, we have for $x$ in the range $X<x<2X$
\begin{equation}\label{eq:3.3}
\big|T\big(x^{2}\big)\big|\ll\sum_{1\,\leq\,h\,\leq\,X}\frac{1}{h}\bigg|\sum_{0\,\leq\,m\,\leq\,x^{2}}r_{{\scriptscriptstyle2}}(m)\exp{\Big(2\pi ih\big(x^{4}-m^{2}\big)^{1/2}\Big)}\bigg|+X\,.
\end{equation}
Applying Cauchy–Schwarz inequality we obtain
\begin{equation}\label{eq:3.4}
T^{2}\big(x^{2}\big)\ll(\log{X})\sum_{1\,\leq\,h\,\leq\,X}\frac{1}{h}\big|S_{h}\big(x^{2}\big)\big|^{2}+X^{2}\,,
\end{equation}
where for $Y>0$ and $h\geq1$ an integer, $S_{h}(Y)$ is given by
\begin{equation*}
S_{h}(Y)=\sum_{0\,\leq\,m\,\leq\,Y}r_{{\scriptscriptstyle2}}(m)\exp{\Big(2\pi ih\big(Y^{2}-m^{2}\big)^{1/2}\Big)}\,.
\end{equation*}
Fix an integer $1\leq h\leq X$. Making a change of variable, we have
\begin{equation}\label{eq:3.5}
\begin{split}
\frac{1}{X}\int\limits_{X}^{2X}\big|S_{h}\big(x^{2}\big)\big|^{2}\textit{d}x&\ll\frac{1}{X^{4}}\int\limits_{X^{4}}^{16X^{4}}\big|S_{h}\big(\sqrt{x}\,\big)\big|^{2}\textit{d}x\\
&=\sum_{0\,\leq\,m,\,n\,\leq\,4X^{2}}r_{{\scriptscriptstyle2}}(m)r_{{\scriptscriptstyle2}}(n)\text{\Large{I}}_{h}(m,n)\,,
\end{split}
\end{equation}
where $\text{\Large{I}}_{h}(m,n)$ is given by
\begin{equation}\label{eq:3.6}
\text{\Large{I}}_{h}(m,n)=\frac{1}{X^{4}}\int\limits_{\text{max}\{m^{2},\,n^{2},\,X^{4}\}}^{16X^{4}}\exp{\Big(2\pi ih\Big\{\big(x-m^{2}\big)^{1/2}-\big(x-n^{2}\big)^{1/2}\Big\}\Big)}\textit{d}x\,.
\end{equation}
We have the estimate
\begin{equation}\label{eq:3.7}
\big|\text{\Large{I}}_{h}(m,n)\big|\ll\left\{
        \begin{array}{ll}
            1 & ;\,m=n 
            \\\\
            \frac{X^{2}}{h|m^{2}-n^{2}|} & ;\, m\neq n\,.
        \end{array}
    \right.
\end{equation}
Inserting \eqref{eq:3.7} into the RHS of \eqref{eq:3.5}, we obtain
\begin{equation}\label{eq:3.8}
\begin{split}
\frac{1}{X}\int\limits_{X}^{2X}\big|S_{h}\big(x^{2}\big)\big|^{2}\textit{d}x&\ll\sum_{0\,\leq\,m\,\leq\,4X^{2}}r^{2}_{{\scriptscriptstyle2}}(m)+X^{2}h^{-1}\sum_{0\,\leq\,m\neq n\,\leq\,4X^{2}}\frac{r_{{\scriptscriptstyle2}}(m)r_{{\scriptscriptstyle2}}(n)}{\big|m^{2}-n^{2}\big|}\\
&\ll X^{2}\log{X}+X^{2}h^{-1}\big(\log{X}\big)^{3}\,.
\end{split}
\end{equation}
Integrating both sides of \eqref{eq:3.4} and using \eqref{eq:3.8}, we find that
\begin{equation}\label{eq:3.9}
\frac{1}{X}\int\limits_{X}^{2X}T^{2}\big(x^{2}\big)\textit{d}x\ll X^{2}\big(\log{X}\big)^{4}\,.
\end{equation}
This concludes the proof.
\end{proof}
\noindent
We end this subsection by quoting the following result (see \cite{montgomery1974hilbert}) which will be needed in subsequent sections of the paper.
\begin{HI}[\cite{montgomery1974hilbert}]
Let $\big(\alpha(\lambda)\big)_{\lambda\in\Lambda}$ and $\big(\beta(\lambda)\big)_{\lambda\in\Lambda}$ be two sequences of complex numbers indexed by a finite set $\Lambda$ of real numbers. Then
\begin{equation*}
\bigg|\,\underset{\lambda\neq\nu}{\sum_{\lambda,\nu\in\Lambda}}\frac{\alpha(\lambda)\overline{\beta(\nu)}}{\lambda-\nu}\bigg|\ll\bigg(\sum_{\lambda\in\Lambda}|\alpha(\lambda)|^{2}\delta_{\lambda}^{-1}\bigg)^{1/2}\bigg(\sum_{\lambda\in\Lambda}|\beta(\lambda)|^{2}\delta_{\lambda}^{-1}\bigg)^{1/2}
\end{equation*}
where $\delta_{\lambda}=\underset{\nu\neq \lambda}{\underset{\nu\in\Lambda}{\textit{min}}}\,|\lambda-\nu|$, and the implied constant is absolute.
\end{HI}
\subsection{Proof of Proposition $2$}We now turn to the proof of Proposition $2$.
\begin{proof}(Proposition $2$).
Fix an integer $M\geq1$, and let $X>M^{1/2}$ be large. In the range $X<x<2X$, we have by Proposition $1$
\begin{equation}\label{eq:3.10}
\begin{split}
\Big|\widehat{\mathcal{E}}_{1}(x)&-\sum_{m\,\leq\,M}\upphi_{{\scriptscriptstyle 1,m}}\big(\sqrt{m}x\big)\Big|\\
&\ll_{\epsilon}\big|W_{{\scriptscriptstyle M,X^{2}}}(x)\big|+x^{-2}\big|T\big(x^{2}\big)\big|+\sum_{m\,\leq\,M}\frac{1}{m}\sum_{k>X/\sqrt{m}}\frac{r_{{\scriptscriptstyle2}}\big(mk^{2}\big)}{k^{2}}+X^{-1+\epsilon}\\
&\ll_{\epsilon}\big|W_{{\scriptscriptstyle M,X^{2}}}(x)\big|+x^{-2}\big|T\big(x^{2}\big)\big|+M^{1/2}X^{-1+\epsilon}\,,
\end{split}
\end{equation}
where $W_{{\scriptscriptstyle M,X^{2}}}(x)$ is given by
\begin{equation}\label{eq:3.11}
W_{{\scriptscriptstyle M,X^{2}}}(x)=\sum_{m\,\leq\,X^{2}}\mathbf{a}(m)\exp{\big(2\pi i\sqrt{m}x\big)}\,,
\end{equation}
and
\begin{equation*}
\mathbf{a}(m)=\left\{
        \begin{array}{ll}
            r_{{\scriptscriptstyle2}}(m)/m & ;\,m=\ell k^{2}\text{ with } \ell>M\text{ square-free}
            \\\\
            0 & ;\,\text{otherwise}\,.
        \end{array}
    \right.
\end{equation*}
Integrating both sides of \eqref{eq:3.10}, we have by  Lemma $4$ and Cauchy–Schwarz inequality
\begin{equation}\label{eq:3.12}
\begin{split}
\frac{1}{X}\int\limits_{X}^{2X}\Big|\widehat{\mathcal{E}}_{1}(x)-\sum_{m\,\leq\,M}\upphi_{{\scriptscriptstyle 1,m}}\big(\sqrt{m}x\big)\Big|\textit{d}x&\ll_{\epsilon}\frac{1}{X}\int\limits_{X}^{2X}\big|W_{{\scriptscriptstyle M,X^{2}}}(x)\big|\textit{d}x+X^{-1}\big(\log{X}\big)^{2}+M^{1/2}X^{-1+\epsilon}\\
&\ll_{\epsilon}\frac{1}{X}\int\limits_{X}^{2X}\big|W_{{\scriptscriptstyle M,X^{2}}}(x)\big|\textit{d}x+M^{1/2}X^{-1+\epsilon}
\end{split}
\end{equation}
It remains to estimate the first term appearing on the RHS of \eqref{eq:3.12}. We have
\begin{equation}\label{eq:3.13}
\begin{split}
\frac{1}{X}\int\limits_{X}^{2X}\big|W_{{\scriptscriptstyle M,X^{2}}}&(x)\big|^{2}\textit{d}x\\
&=\sum_{m\,\leq\,X^{2}}\mathbf{a}^{2}(m)+\frac{1}{2\pi iX}\bigg\{\sum_{m\neq n\,\leq\,X^{2}}\frac{\mathbf{a}_{{\scriptscriptstyle2X}}(m)\overline{\mathbf{a}_{{\scriptscriptstyle2X}}(n)}}{\sqrt{m}-\sqrt{n}}-\sum_{m\neq n\,\leq\,X^{2}}\frac{\mathbf{a}_{{\scriptscriptstyle X}}(m)\overline{\mathbf{a}_{{\scriptscriptstyle X}}(n)}}{\sqrt{m}-\sqrt{n}}\bigg\}\,,
\end{split}
\end{equation}
where for $\gamma>0$ and $m\geq1$ an integer, we define $\mathbf{a}_{{\scriptscriptstyle\gamma}}(m)=\mathbf{a}(m)\exp{\big(2\pi i\sqrt{m}\gamma\big)}$. We first estimate the off-diagonal terms. By Hilbert's inequality, we have for $\gamma=X, 2X$
\begin{equation}\label{eq:3.14}
\begin{split}
\bigg|\,\,\sum_{m\neq n\,\leq\,X^{2}}\frac{\mathbf{a}_{{\scriptscriptstyle\gamma}}(m)\overline{\mathbf{a}_{{\scriptscriptstyle\gamma}}(n)}}{\sqrt{m}-\sqrt{n}}\,\,\bigg|&\ll\sum_{m\,\leq\,X^{2}}\mathbf{a}^{2}(m)m^{1/2}\\
&\ll\sum_{m\,>M}r^{2}_{{\scriptscriptstyle2}}(m)m^{-3/2}\ll M^{-1/2}\log{2M}\,.
\end{split}
\end{equation}
Inserting \eqref{eq:3.14} into the RHS of \eqref{eq:3.13} and recalling that $X>M^{1/2}$, it follows that
\begin{equation}\label{eq:3.15}
\begin{split}
\frac{1}{X}\int\limits_{X}^{2X}\big|W_{{\scriptscriptstyle M,X^{2}}}(x)\big|^{2}\textit{d}x&\ll\sum_{m\,>M}r^{2}_{{\scriptscriptstyle2}}(m)m^{-2}+M^{-1}\log{2M}\\
&\ll M^{-1}\log{2M}\,.
\end{split}
\end{equation}
Inserting \eqref{eq:3.15} into the RHS of \eqref{eq:3.12}, applying Cauchy–Schwarz inequality and then taking $\limsup$, we arrive at
\begin{equation}\label{eq:3.16}
\begin{split}
\limsup_{X\to\infty}\frac{1}{X}\int\limits_{X}^{2X}\Big|\widehat{\mathcal{E}}_{1}(x)-\sum_{m\,\leq\,M}\upphi_{{\scriptscriptstyle 1,m}}\big(\sqrt{m}x\big)\Big|\textit{d}x\ll  M^{-1/2}\big(\log{2M}\big)^{1/2}
\end{split}
\end{equation}
Finally, letting $M\to\infty$ in \eqref{eq:3.16} concludes the proof.
\end{proof}
\section{The probability density}
\noindent
Having proved Proposition $2$ in the last section, in this section we turn to the construction of the probability density $\mathcal{P}_{1}(\alpha)$. We begin by making the following definition.
\begin{mydef}
For $\alpha\in\mathbb{C}$ and $M\geq1$ an integer, define 
\begin{equation*}
\mathfrak{M}_{X}(\alpha;M)=\frac{1}{X}\int\limits_{X}^{2X}\exp{\bigg(2\pi i\alpha\sum_{m\,\leq\,M}\upphi_{{\scriptscriptstyle 1,m}}\big(\sqrt{m}x\big)\bigg)}\textit{d}x\,,
\end{equation*}
and let
\begin{equation*}
\mathfrak{M}(\alpha;M)=\prod_{m\,\leq\,M}\Phi_{{\scriptscriptstyle1,m}}(\alpha)\quad;\quad\Phi_{{\scriptscriptstyle1,m}}(\alpha)=\int\limits_{0}^{1}\exp{\big(2\pi i\alpha\upphi_{{\scriptscriptstyle 1,m}}(t)\big)}\textit{d}t\,,
\end{equation*}
\end{mydef}
\noindent
We quote the following result (see \cite{heath1992distribution}, Lemma $2$) which will be needed in subsequent sections of the paper.
\begin{lemma}[\cite{heath1992distribution}]
Suppose that $\mathbf{b}_{{\scriptscriptstyle 1}}(t), \mathbf{b}_{{\scriptscriptstyle 2}}(t),\ldots, \mathbf{b}_{{\scriptscriptstyle k}}(t)$ are continuous functions from $\mathbb{R}$ to $\mathbb{C}$, periodic of period $1$, and that $\gamma_{{\scriptscriptstyle 1}}, \gamma_{{\scriptscriptstyle 2}},\ldots, \gamma_{{\scriptscriptstyle k}}$ are positive real numbers which are linearly independent over $\mathbb{Q}$. Then 
\begin{equation}\label{eq:4.1}
\lim_{X\to\infty}\frac{1}{X}\int\limits_{X}^{2X}\prod_{i\,\leq\,k}\mathbf{b}_{{\scriptscriptstyle i}}(\gamma_{{\scriptscriptstyle i}}x)\textit{d}x=\prod_{i\,\leq\,k}\int\limits_{0}^{1}\mathbf{b}_{{\scriptscriptstyle i}}(t)\textit{d}t\,.
\end{equation}
\end{lemma}
\noindent
We now state the main result of this section.
\begin{prop}
\text{ }\\
(I) We have
\begin{equation}\label{eq:4.2}
\lim_{X\to\infty}\mathfrak{M}_{X}(\alpha;M)=\mathfrak{M}(\alpha;M)\,.
\end{equation}
More over, if we let
\begin{equation*}
\Phi_{{\scriptscriptstyle1}}(\alpha)=\prod_{m=1}^{\infty}\Phi_{{\scriptscriptstyle1,m}}(\alpha)\,,
\end{equation*}
then $\Phi_{{\scriptscriptstyle1}}(\alpha)$ defines an entire function of $\alpha$, where the infinite product converges absolutely and uniformly on any compact subset of the plane. For large $|\alpha|$, $\alpha=\sigma+i\tau$, $\Phi_{{\scriptscriptstyle1}}(\alpha)$ satisfies the bound
\begin{equation}\label{eq:4.3}
\big|\Phi_{{\scriptscriptstyle1}}(\alpha)\big|\leq\exp{\bigg(-\frac{\pi^{2}}{2}\big(C_{{\scriptscriptstyle1}}^{-1}\sigma^{2}-C_{{\scriptscriptstyle1}}\tau^{2}\big)|\alpha|^{-1/\theta(|\alpha|)}\log{|\alpha|}+C_{{\scriptscriptstyle1}}|\tau|\log{|\alpha|}\bigg)}\,,
\end{equation}
where $C_{{\scriptscriptstyle1}}>1$ is an absolute constant, and $\theta(x)=1-c/\log\log{x}$ with $c>0$ an absolute constant.\\
%
%
%
%
(II) For $x\in\mathbb{R}$, let $\mathcal{P}_{1}(x)=\widehat{\Phi}_{{\scriptscriptstyle1}}(x)$ be the Fourier transform of $\Phi_{{\scriptscriptstyle1}}$. Then $\mathcal{P}_{1}(x)$ defines a probability density which satisfies for any non-negative integer $j\geq0$ and any $x\in\mathbb{R}$, $|x|$ sufficiently large in terms of $j$, the bound
\begin{equation}\label{eq:4.4}
\big|\mathcal{P}^{(j)}_{1}(x)\big|\leq\exp{\bigg(-\frac{\pi}{2}|x|\exp{\big(\rho|x|\big)}\bigg)}\quad;\quad\rho=\frac{\pi}{5C_{{\scriptscriptstyle1}}}\,.
\end{equation}
\end{prop}
\begin{proof}
We begin with the proof of $\textit{(I)}$. Let $\alpha\in\mathbb{C}$ and $M\geq1$ an integer. We are going apply Lemma $5$ with $\mathbf{b}_{{\scriptscriptstyle m}}(t)=\exp{\big(2\pi i\alpha\upphi_{{\scriptscriptstyle 1,m}}(t)\big)}$, and frequencies $\gamma_{{\scriptscriptstyle m}}=\sqrt{m}$. The elements of the set $\mathscr{B}=\{\sqrt{m}:|\mu(m)|=1\}$ are are linearly independent over $\mathbb{Q}$, and since $\upphi_{{\scriptscriptstyle 1,m}}(t)\equiv0$ whenever $\sqrt{m}\notin\mathscr{B}$ (i.e. whenever $m$ is not square-free), we see that the conditions of Lemma $5$ are satisfied, and thus \eqref{eq:4.2} holds.\\
Now, let us show that $\Phi_{{\scriptscriptstyle1}}(\alpha)$ defines an entire function of $\alpha$. First we note the following. By the definition of $\upphi_{{\scriptscriptstyle 1,m}}(t)$, we have
\begin{equation}\label{eq:4.5}
\begin{split}
\int\limits_{0}^{1}\upphi_{{\scriptscriptstyle 1,m}}(t)\textit{d}t&=-\frac{\sqrt{2}}{\pi}\frac{\mu^{2}(m)}{m}\sum_{k=1}^{\infty}\frac{r_{{\scriptscriptstyle2}}\big(mk^{2}\big)}{k^{2}}\int\limits_{0}^{1}\cos{(2\pi kt)}\textit{d}t\\
&=0\,,
\end{split}
\end{equation}
and we also have the uniform bound
\begin{equation}\label{eq:4.6}
\begin{split}
\big|\upphi_{{\scriptscriptstyle 1,m}}(t)\big|&\leq\frac{\sqrt{2}}{\pi}\frac{\mu^{2}(m)}{m}\sum_{k=1}^{\infty}\frac{r_{{\scriptscriptstyle2}}\big(mk^{2}\big)}{k^{2}}\\
&\leq2^{5/2}\pi^{-1}\mu^{2}(m)\frac{r_{{\scriptscriptstyle2}}(m)}{m}\sum_{k=1}^{\infty}\frac{r_{{\scriptscriptstyle2}}\big(k^{2}\big)}{k^{2}}\leq\mathfrak{a}\,\mu^{2}(m)\frac{r_{{\scriptscriptstyle2}}(m)}{m}\,,
\end{split}
\end{equation}
for some absolute constant $\mathfrak{a}>0$. By \eqref{eq:4.5} and \eqref{eq:4.6} we obtain
\begin{equation}\label{eq:4.7}
\begin{split}
\Phi_{{\scriptscriptstyle1,m}}(\alpha)&=\int\limits_{0}^{1}\exp{\big(2\pi i\alpha\upphi_{{\scriptscriptstyle 1,m}}(t)\big)}\textit{d}t\\
&=1+O\bigg(|\alpha|\frac{r_{{\scriptscriptstyle2}}(m)}{m}\bigg)^{2}\,,
\end{split}
\end{equation}
whenever, say, $m\geq|\alpha|^{2}$. Since $\sum_{m=1}^{\infty}r^{2}_{{\scriptscriptstyle2}}(m)/m^{2}<\infty$, it follows from \eqref{eq:4.7} that the infinite product $\prod_{m=1}^{\infty}\Phi_{{\scriptscriptstyle1,m}}(\alpha)$ converges absolutely and uniformly on any compact subset of the plane, and so $\Phi_{{\scriptscriptstyle1}}(\alpha)$ defines an entire function of $\alpha$.\\
We now estimate $\Phi_{{\scriptscriptstyle1}}(\alpha)$ for large $|\alpha|$, $\alpha=\sigma+i\tau$. Let $c>0$ be an absolute constant such that (see \cite{Wigert})
\begin{equation}\label{eq:4.8}
r_{{\scriptscriptstyle2}}(m)\leq m^{c/\log\log{m}}\quad:\quad m>2\,,
\end{equation}
and for real $x>\textit{e}$ we write $\theta(x)=1-c/\log\log{x}$. Let $\epsilon>0$ be a small absolute constant which will be specified later, and set
\begin{equation*}
\ell=\ell(\alpha)=\Big[\Big(\epsilon^{-1}|\alpha|\Big)^{1/\theta(|\alpha|)}\Big]+1\,.
\end{equation*}
We are going to estimate the infinite product $\prod_{m=1}^{\infty}\Phi_{{\scriptscriptstyle1,m}}(\alpha)$ separately for $m<\ell$ and $m\geq\ell$. In what follows, we assume that $|\alpha|$ is sufficiently large in terms of $c$ and $\epsilon$. The product over $m<\ell$ is estimated trivially by using the upper bound \eqref{eq:4.6}, leading to
\begin{equation}\label{eq:4.9}
\begin{split}
\Big|\prod_{m<\ell}\Phi_{{\scriptscriptstyle1,m}}(\alpha)\Big|&\leq\exp{\bigg(2\pi\mathfrak{a}|\tau|\sum_{m<\ell}\mu^{2}(m)\frac{r_{{\scriptscriptstyle2}}(m)}{m}\bigg)}\\
&\leq\exp{\Big(\mathfrak{b}|\tau|\log{|\alpha|}\Big)}\,,
\end{split}
\end{equation}
for some absolute constant $\mathfrak{b}>0$.\\
Suppose now that $m\geq\ell$. By \eqref{eq:4.8} we have
\begin{equation}\label{eq:4.10}
\frac{r_{{\scriptscriptstyle2}}(m)}{m}|\alpha|\leq m^{-\theta(m)}|\alpha|\leq\big(\epsilon^{-1}|\alpha|\big)^{-\frac{\theta(m)}{\theta(|\alpha|)}}|\alpha|\leq\epsilon
\end{equation}
It follows from \eqref{eq:4.5}, \eqref{eq:4.6} and \eqref{eq:4.10} that
\begin{equation}\label{eq:4.11}
\Phi_{{\scriptscriptstyle1,m}}(\alpha)=1-\frac{(2\pi\alpha)^{2}}{2}\int\limits_{0}^{1}\upphi^{2}_{{\scriptscriptstyle 1,m}}(t)\textit{d}t+\mathcal{R}_{{\scriptscriptstyle m}}(\alpha)\,,
\end{equation}
where the remainder term $\mathcal{R}_{{\scriptscriptstyle m}}(\alpha)$ satisfies the bound
\begin{equation}\label{eq:4.12}
\big|\mathcal{R}_{{\scriptscriptstyle m}}(\alpha)\big|\leq\bigg\{\frac{2}{3}\pi\mathfrak{a}\epsilon\exp{\big(2\pi\mathfrak{a}\epsilon\big)}\bigg\}\frac{(2\pi|\alpha|)^{2}}{2}\int\limits_{0}^{1}\upphi^{2}_{{\scriptscriptstyle 1,m}}(t)\textit{d}t\,.
\end{equation}
At this point we specify $\epsilon$, by choosing $0<\epsilon<\frac{1}{64}$ such that
\begin{equation*}
2\pi\mathfrak{a}\epsilon^{1/2}\exp{\big(2\pi\mathfrak{a}\epsilon\big)}\leq1\,.
\end{equation*}
With this choice of $\epsilon$, we have
\begin{equation}\label{eq:4.13}
\big|\mathcal{R}_{{\scriptscriptstyle m}}(\alpha)\big|\leq\epsilon^{1/2}\frac{(2\pi|\alpha|)^{2}}{2}\int\limits_{0}^{1}\upphi^{2}_{{\scriptscriptstyle 1,m}}(t)\textit{d}t\,,
\end{equation}
and we also note that $\big|\Phi_{{\scriptscriptstyle1,m}}(\alpha)-1\big|\leq\epsilon<\frac{1}{2}$. Thus, on rewriting \eqref{eq:4.11} in the form
\begin{equation}\label{eq:4.14}
\Phi_{{\scriptscriptstyle1,m}}(\alpha)=\exp{\Bigg(-\frac{(2\pi\alpha)^{2}}{2}\int\limits_{0}^{1}\upphi^{2}_{{\scriptscriptstyle 1,m}}(t)\textit{d}t+\widetilde{\mathcal{R}}_{{\scriptscriptstyle m}}(\alpha)\Bigg)}\,,
\end{equation}
it follows from \eqref{eq:4.13} that the remainder term $\widetilde{\mathcal{R}}_{{\scriptscriptstyle m}}(\alpha)$ satisfies the bound
\begin{equation}\label{eq:4.15}
\big|\widetilde{\mathcal{R}}_{{\scriptscriptstyle m}}(\alpha)\big|\leq\epsilon^{1/2}(2\pi|\alpha|)^{2}\int\limits_{0}^{1}\upphi^{2}_{{\scriptscriptstyle 1,m}}(t)\textit{d}t\,.
\end{equation}
From \eqref{eq:4.14} and \eqref{eq:4.15} we obtain
\begin{equation}\label{eq:4.16}
\prod_{m\geq\ell}\big|\Phi_{{\scriptscriptstyle1,m}}(\alpha)\big|\leq\exp{\Bigg(-\frac{\pi^{2}}{2}\big(3\sigma^{2}-5\tau^{2}\big)\sum_{m\,\geq\,\ell}\,\int\limits_{0}^{1}\upphi^{2}_{{\scriptscriptstyle 1,m}}(t)\textit{d}t\Bigg)}\,.
\end{equation}
Now, by \eqref{eq:4.6} we have
\begin{equation}\label{eq:4.17}
\sum_{m\,\geq\,\ell}\,\int\limits_{0}^{1}\upphi^{2}_{{\scriptscriptstyle 1,m}}(t)\textit{d}t\leq\mathfrak{a}^{2}\sum_{m\,\geq\,\ell}\mu^{2}(m)\frac{r^{2}_{{\scriptscriptstyle2}}(m)}{m^{2}}\,,
\end{equation}
and since
\begin{equation}\label{eq:4.18}
\begin{split}
\int\limits_{0}^{1}\upphi^{2}_{{\scriptscriptstyle 1,m}}(t)\textit{d}t&=\pi^{-2}\frac{\mu^{2}(m)}{m^{2}}\sum_{k=1}^{\infty}\frac{r^{2}_{{\scriptscriptstyle2}}\big(mk^{2}\big)}{k^{4}}\\
&\geq\pi^{-2}\mu^{2}(m)\frac{r^{2}_{{\scriptscriptstyle2}}(m)}{m^{2}}\,,
\end{split}
\end{equation}
we also have the lower bound
\begin{equation}\label{eq:4.19}
\sum_{m\,\geq\,\ell}\,\int\limits_{0}^{1}\upphi^{2}_{{\scriptscriptstyle 1,m}}(t)\textit{d}t\geq\pi^{-2}\sum_{m\,\geq\,\ell}\mu^{2}(m)\frac{r^{2}_{{\scriptscriptstyle2}}(m)}{m^{2}}\,.
\end{equation}
Since
\begin{equation}\label{eq:4.20}
\sum_{m\,\leq\,Y}\mu^{2}(m)r^{2}_{{\scriptscriptstyle2}}(m)\sim\mathfrak{h}Y\log{Y}\,,
\end{equation}
as $Y\to\infty$ where $\mathfrak{h}>0$ is some constant, we obtain by partial summation
\begin{equation}\label{eq:4.21}
\big(3\sigma^{2}-5\tau^{2}\big)\sum_{m\,\geq\,\ell}\,\int\limits_{0}^{1}\upphi^{2}_{{\scriptscriptstyle 1,m}}(t)\textit{d}t\geq\big(A\sigma^{2}-B\tau^{2}\big)|\alpha|^{-1/\theta(|\alpha|)}\log{|\alpha|}\,,
\end{equation}
for some absolute constants $A,B>0$. Inserting \eqref{eq:4.21} into the RHS of \eqref{eq:4.16} we arrive at
\begin{equation}\label{eq:4.22}
\prod_{m\geq\ell}\big|\Phi_{{\scriptscriptstyle1,m}}(\alpha)\big|\leq\exp{\bigg(-\frac{\pi^{2}}{2}\big(A\sigma^{2}-B\tau^{2}\big)|\alpha|^{-1/\theta(|\alpha|)}\log{|\alpha|}\bigg)}\,.
\end{equation}
Finally, setting $C_{{\scriptscriptstyle1}}=1+\text{max}\{A, A^{-1}, B, \mathfrak{b}\}$, we deduce from \eqref{eq:4.9} and \eqref{eq:4.22}
\begin{equation}\label{eq:4.23}
\big|\Phi_{{\scriptscriptstyle1}}(\alpha)\big|\leq\exp{\bigg(-\frac{\pi^{2}}{2}\big(C_{{\scriptscriptstyle1}}^{-1}\sigma^{2}-C_{{\scriptscriptstyle1}}\tau^{2}\big)|\alpha|^{-1/\theta(|\alpha|)}\log{|\alpha|}+C_{{\scriptscriptstyle1}}|\tau|\log{|\alpha|}\bigg)}\,.
\end{equation}
%
%
%
%
%
%
%
%
%
%
This completes the proof of $\textit{(I)}$.\\
We now turn to the proof of $\textit{(II)}$. By the definition of the Fourier transform, we have for real $x$
\begin{equation}\label{eq:4.24}
\mathcal{P}_{1}(x)=\int\limits_{-\infty}^{\infty}\Phi_{{\scriptscriptstyle1}}(\sigma)\exp{\big(-2\pi ix\sigma\big)}\textit{d}\sigma\,,
\end{equation}
%
%
%
%
%
%
and it follows from the decay estimates \eqref{eq:4.23} that $\mathcal{P}_{1}(x)$ is of class $C^{\infty}$. Let us estimate $|\mathcal{P}^{(j)}_{1}(x)|$ for real $x$, $|x|$ sufficiently large in terms $j$. Let $\tau=\tau_{x}$, $|\tau|$ large, be a real number depending $x$ to be determined later, which satisfies $\textit{sgn}(\tau)=-\text{sgn}(x)$. By Cauchy's theorem and the decay estimate \eqref{eq:4.23}, we have
\begin{equation}\label{eq:4.25}
\begin{split}
\mathcal{P}^{(j)}_{1}(x)&=
\big(-2\pi i\big)^{j}\int\limits_{-\infty}^{\infty}\sigma^{j}\Phi_{{\scriptscriptstyle1}}(\sigma)\exp{\big(-2\pi ix\sigma\big)}\textit{d}\sigma\\
&=\big(-2\pi i\big)^{j}\exp{\big(-2\pi |x||\tau|\big)}\int\limits_{-\infty}^{\infty}\big(\sigma+i\tau\big)^{j}\Phi_{{\scriptscriptstyle1}}(\sigma+i\tau)\exp{\big(-2\pi ix\sigma\big)}\textit{d}\sigma\,.
\end{split}
\end{equation}
It follows that
\begin{equation}\label{eq:4.26}
\begin{split}
&\big|\mathcal{P}^{(j)}_{1}(x)\big|\leq\big(2\pi C_{{\scriptscriptstyle1}}|\tau|\big)^{j+1}\exp{\big(-2\pi |x||\tau|\big)}\int\limits_{-\infty}^{\infty}\big(\sigma^{2}+1\big)^{j/2}\big|\Phi_{{\scriptscriptstyle1}}\big(C_{{\scriptscriptstyle1}}\tau\sigma+i\tau\big)\big|\textit{d}\sigma\,.
\end{split}
\end{equation}
We decompose the range of integration in \eqref{eq:4.26} as follows
\begin{equation}\label{eq:4.27}
\begin{split}
\int\limits_{-\infty}^{\infty}\big(\sigma^{2}+1\big)^{j/2}\big|\Phi_{{\scriptscriptstyle1}}\big(C_{{\scriptscriptstyle1}}\tau\sigma+i\tau\big)\big|\textit{d}\sigma&=\int\limits_{|\sigma|\leq\sqrt{2}}\ldots\textit{d}\sigma+\int\limits_{|\sigma|>\sqrt{2}}\ldots\textit{d}\sigma\\
&=\text{\Large{L}}_{1}+\text{\Large{L}}_{2}\,.
\end{split}
\end{equation}
In what follows, we assume that $|\tau|$ is sufficiently large in terms of $j$. In the range $|\sigma|\leq\sqrt{2}$, we have by \eqref{eq:4.23}
\begin{equation}\label{eq:4.28}
\big|\Phi_{{\scriptscriptstyle1}}\big(C_{{\scriptscriptstyle1}}\tau\sigma+i\tau\big)\big|\leq\exp{\big(4C_{{\scriptscriptstyle1}}|\tau|\log{|\tau|}\big)}\,.
\end{equation}
From \eqref{eq:4.28}, it follows that $\text{\Large{L}}_{1}$ satisfies the bound 
\begin{equation}\label{eq:4.29}
\text{\Large{L}}_{1}\leq\exp{\big(5C_{{\scriptscriptstyle1}}|\tau|\log{|\tau|}\big)}\,.
\end{equation}
Referring to \eqref{eq:4.23} once again, we have in the range $|\sigma|>\sqrt{2}$
\begin{equation}\label{eq:4.30}
\big|\Phi_{{\scriptscriptstyle1}}\big(C_{{\scriptscriptstyle1}}\tau\sigma+i\tau\big)\big|\leq\exp{\bigg(-C_{{\scriptscriptstyle1}}|\tau|\bigg\{\,\bigg|\bigg(\frac{\pi^{2}}{16C^{2}_{{\scriptscriptstyle1}}}\bigg)^{\frac{\theta(|\tau|)}{2\theta(|\tau|)-1}}|\tau|^{\frac{\theta(|\tau|)-1}{2\theta(|\tau|)-1}}\sigma\bigg|^{\frac{2\theta(|\tau|)-1}{\theta(|\tau|)}}-1\bigg\}\log{|C_{{\scriptscriptstyle1}}\tau\sigma+i\tau|}\bigg)}
\end{equation}
From \eqref{eq:4.30}, it follows that
\begin{equation}\label{eq:4.31}
\text{\Large{L}}_{2}\leq2^{j/2}\Tilde{\tau}^{j+1}\int\limits_{-\infty}^{\infty}|\sigma|^{j}\exp\bigg(-C_{{\scriptscriptstyle1}}|\tau|\Big\{|\sigma|^{\frac{2\theta(|\tau|)-1}{\theta(|\tau|)}}-1\Big\}\log{|C_{{\scriptscriptstyle1}}\tau\Tilde{\tau}\sigma+i\tau|}\bigg)\textit{d}\sigma\,,
\end{equation}
where $\Tilde{\tau}$ is given by
\begin{equation*}
\Tilde{\tau}=\bigg(\frac{16C^{2}_{{\scriptscriptstyle1}}}{\pi^{2}}\bigg)^{\frac{\theta(|\tau|)}{2\theta(|\tau|)-1}}|\tau|^{\frac{1-\theta(|\tau|)}{2\theta(|\tau|)-1}}\,.
\end{equation*}
We decompose the range of integration in \eqref{eq:4.31} as follows
\begin{equation}\label{eq:4.32}
\begin{split}
\int\limits_{-\infty}^{\infty}|\sigma|^{j}\exp\bigg(-C_{{\scriptscriptstyle1}}|\tau|\Big\{|\sigma|^{\frac{2\theta(|\tau|)-1}{\theta(|\tau|)}}-1\Big\}\log{|C_{{\scriptscriptstyle1}}\tau\Tilde{\tau}\sigma+i\tau|}\bigg)\textit{d}\sigma&=\int\limits_{|\sigma|\leq u}\ldots\textit{d}\sigma+\int\limits_{|\sigma|>u}\ldots\textit{d}\sigma\\
&=\text{\Large{L}}_{3}+\text{\Large{L}}_{4}\,,
\end{split}
\end{equation}
where $u=2^{\frac{\theta(|\tau|)}{2\theta(|\tau|)-1}}$. In the range $|\sigma|\leq u$ we estimate trivially, obtaining
\begin{equation}\label{eq:4.33}
\text{\Large{L}}_{3}\leq\exp{\big(3C_{{\scriptscriptstyle1}}|\tau|\log{|\tau|}\big)}\,.
\end{equation}
In the range $|\sigma|>u$ we have $|\sigma|^{\frac{2\theta(|\tau|)-1}{\theta(|\tau|)}}-1\geq\frac{1}{2}|\sigma|^{\frac{2\theta(|\tau|)-1}{\theta(|\tau|)}}\geq\frac{1}{2}|\sigma|^{1/2}$. It follows that 
\begin{equation}\label{eq:4.34}
\text{\Large{L}}_{4}\leq\int\limits_{-\infty}^{\infty}|\sigma|^{j}\exp\big(-|\sigma|^{1/2}\,\big)\textit{d}\sigma\,.
\end{equation}
Combining \eqref{eq:4.33} and \eqref{eq:4.34}, we see that  $\text{\Large{L}}_{2}$ satisfies the bound
\begin{equation}\label{eq:4.35}
\text{\Large{L}}_{2}\leq2^{j/2+1}\Tilde{\tau}^{j+1}\exp{\big(3C_{{\scriptscriptstyle1}}|\tau|\log{|\tau|}\big)}\,.
\end{equation}
From \eqref{eq:4.29} and \eqref{eq:4.35} we find that $\text{\Large{L}}_{1}$ dominates. By \eqref{eq:4.26} and \eqref{eq:4.27} we arrive at
\begin{equation}\label{eq:4.36}
\begin{split}
&\big|\mathcal{P}^{(j)}_{1}(x)\big|\leq\big(4\pi C_{{\scriptscriptstyle1}}|\tau|\big)^{j+1}\exp{\bigg(-2\pi|\tau|\bigg\{|x|-\frac{5C_{{\scriptscriptstyle1}}}{2\pi}\log{|\tau|}\bigg\}\bigg)\,}
\end{split}
\end{equation}
Finally, we specify $\tau$. We choose
\begin{equation*}
\tau=-\text{sgn}(x)\exp{\big(\rho|x|\big)}\quad;\quad\rho=\frac{\pi}{5C_{{\scriptscriptstyle1}}}\,.
\end{equation*}
With this choice, we have the bound (recall that $|x|$ is assumed to be large in terms of $j$)
\begin{equation}\label{eq:4.37}
\begin{split}
\big|\mathcal{P}^{(j)}_{1}(x)\big|&\leq\big(4\pi C_{{\scriptscriptstyle1}}\big)^{j+1}\exp{\bigg(-\pi|x|\exp{\big(\rho|x|\big)}+(j+1)\rho|x|\bigg)}\\
&\leq\exp{\bigg(-\frac{\pi}{2}|x|\exp{\big(\rho|x|\big)}\bigg)}\,.
\end{split}
\end{equation}
It remains to show that $\mathcal{P}_{1}(x)$ defines a probability density. This will be a consequence of the proof of Theorem 1. This concludes the proof of Proposition $3$.
\end{proof}
\section{Power moment estimates}
\noindent
Having constructed the probability density $\mathcal{P}_{1}(\alpha)$ in the previous section, our final task, before turning to the proof of the main results of this paper, is to establish the existence of all moments of the normalized error term $\widehat{\mathcal{E}}_{1}(x)$. The main result we shall set out to prove is the following.
\begin{prop}
Let $j\geq1$ be an integer. Then the $j$-th power moment of $\widehat{\mathcal{E}}_{1}(x)$ is given by
\begin{equation}\label{eq:5.1}
\lim\limits_{X\to\infty}\frac{1}{X}\int\limits_{ X}^{2X}\widehat{\mathcal{E}}^{j}_{1}(x)\textit{d}x=\sum_{s=1}^{j}\underset{\,\,\ell_{1},\,\ldots\,,\ell_{s}\geq1}{\sum_{\ell_{1}+\cdots+\ell_{s}=j}}\,\,\frac{j!}{\ell_{1}!\cdots\ell_{s}!}\underset{\,\,\,m_{s}>\cdots>m_{1}}{\sum_{m_{1},\,\ldots\,,m_{s}=1}^{\infty}}\prod_{i=1}^{s}\Xi(m_{i},\ell_{i})\,,
\end{equation}
where the series on the RHS of \eqref{eq:5.1} converges absolutely, and for integers $m,\ell\geq1$, the term $\Xi(m,\ell)$ is given as in \eqref{eq:1.4}. 
%
%
%
%
\end{prop}
\noindent
\textbf{Remark.} As we shall see later on, the RHS of \eqref{eq:5.1} is simply $\int_{-\infty}^{\infty}\alpha^{j}\mathcal{P}_{1}(\alpha)\textit{d}\alpha$. Also, our proof of Proposition $4$ in fact yields \eqref{eq:5.1} in a quantitative form, namely, we obtain $\frac{1}{X}\int_{ X}^{2X}\widehat{\mathcal{E}}^{j}_{1}(x)\textit{d}x=\int_{-\infty}^{\infty}\alpha^{j}\mathcal{P}_{1}(\alpha)\textit{d}\alpha+\mathrm{R}_{j}(X)$ with an explicit decay estimate for the remainder term $\mathrm{R}_{j}(X)$. However, as our sole focus here is on establishing the existence of the limit given in the LHS of \eqref{eq:5.1}, Proposition $4$ will suffice for our needs.
%
%
\begin{proof}
We split the proof into three cases, depending on whether $j=1$, $j=2$ or $j\geq3$.
\textbf{Case 1.} $j=1$. Since by definition, $\Xi(m,1)=0$ for any integer $m\geq1$, we need to show that $\frac{1}{X}\int_{X}^{2X}\widehat{\mathcal{E}}_{1}(x)\textit{d}x\to0$ as $X\to\infty$. By Proposition $1$ and Lemma $4$, we have
\begin{equation}\label{eq:5.2}
\frac{1}{X}\int\limits_{ X}^{2X}\widehat{\mathcal{E}}_{1}(x)\textit{d}x=-\frac{\sqrt{2}}{\pi}\sum_{m\,\leq\,X^{2}}\frac{r_{{\scriptscriptstyle 2}}(m)}{m}\frac{1}{X}\int\limits_{ X}^{2X}\cos{\big(2\pi\sqrt{m}x\big)}\textit{d}x+O_{\epsilon}\big(X^{-1+\epsilon}\big)\,.
\end{equation}
It follows that
\begin{equation}\label{eq:5.3}
\begin{split}
\bigg|\frac{1}{X}\int\limits_{ X}^{2X}\widehat{\mathcal{E}}_{1}(x)\textit{d}x\bigg|&\ll_{\epsilon}\frac{1}{X}\sum_{m=1}^{\infty}\frac{r_{{\scriptscriptstyle 2}}(m)}{m^{3/2}}+X^{-1+\epsilon}\\
&\ll_{\epsilon} X^{-1+\epsilon}\,.
\end{split}
\end{equation}
This proves \eqref{eq:5.1} in the case where $j=1$.\\
\textbf{Case 2.} $j=2$. By Proposition $1$ and Lemma $4$, we have
\begin{equation}\label{eq:5.4}
\frac{1}{X}\int\limits_{ X}^{2X}\widehat{\mathcal{E}}^{2}_{1}(x)\textit{d}x=\frac{2}{\pi^{2}}\sum_{m,\,n\,\leq\,X^{2}}\frac{r_{{\scriptscriptstyle 2}}(m)}{m}\frac{r_{{\scriptscriptstyle 2}}(n)}{n}\frac{1}{X}\int\limits_{ X}^{2X}\cos{\big(2\pi\sqrt{m}x\big)}\cos{\big(2\pi\sqrt{n}x\big)}\textit{d}x+O_{\epsilon}\big(X^{-1+\epsilon}\big)\,,
\end{equation}
Now, we have
\begin{equation}\label{eq:5.5}
\begin{split}
\frac{1}{X}\int\limits_{ X}^{2X}\cos{\big(2\pi\sqrt{m}x\big)}\cos{\big(2\pi\sqrt{n}x\big)}\textit{d}x&=\frac{1}{2}\mathds{1}_{m=n}+\frac{1}{4\pi X}\mathds{1}_{m\neq n}\frac{\sin{\big(2\pi(\sqrt{m}-\sqrt{n}\,)x\big)}}{\sqrt{m}-\sqrt{n}}\Bigg|^{x=2X}_{x=X}\\
&\,\,\,\,\,\,+O\bigg(\frac{1}{X(mn)^{1/4}}\bigg)\\
&=\frac{1}{2}\mathds{1}_{m=n}+\frac{1}{4\pi X}\mathds{1}_{m\neq n}\text{\Large{K}}_{(\sqrt{m},\sqrt{n}\,)}+O\bigg(\frac{1}{X(mn)^{1/4}}\bigg)\,,
\end{split}
\end{equation}
say. Inserting \eqref{eq:5.5} into the RHS of \eqref{eq:5.4}, we obtain
\begin{equation}\label{eq:5.6}
\frac{1}{X}\int\limits_{ X}^{2X}\widehat{\mathcal{E}}^{2}_{1}(x)\textit{d}x=\frac{1}{\pi^{2}}\sum_{m\,\leq\,X^{2}}\frac{r^{2}_{{\scriptscriptstyle 2}}(m)}{m^{2}}+\frac{1}{2\pi^{3}X}\sum_{m\neq n\,\leq\,X^{2}}\frac{r_{{\scriptscriptstyle 2}}(m)}{m}\frac{r_{{\scriptscriptstyle 2}}(n)}{n}\text{\Large{K}}_{(\sqrt{m},\sqrt{n}\,)}+O_{\epsilon}\big(X^{-1+\epsilon}\big)\,,
\end{equation}
To estimate the off-diagonal terms, we use the identity $\sin{t}=\frac{1}{2\pi i}\big(\exp{(it)}-\exp{(-it})\big)$ and then apply Hilbert's inequality, obtaining
\begin{equation}\label{eq:5.7}
\bigg|\sum_{m\neq n\,\leq\,X^{2}}\frac{r_{{\scriptscriptstyle 2}}(m)}{m}\frac{r_{{\scriptscriptstyle 2}}(n)}{n}\text{\Large{K}}_{(\sqrt{m},\sqrt{n}\,)}\bigg|\ll\sum_{m=1}^{\infty}r^{2}_{{\scriptscriptstyle2}}(m)m^{-3/2}\,.
\end{equation}
Inserting \eqref{eq:5.7} into the RHS of \eqref{eq:5.6}, we arrive at
\begin{equation}\label{eq:5.8}
\begin{split}
\frac{1}{X}\int\limits_{ X}^{2X}\widehat{\mathcal{E}}^{2}_{1}(x)\textit{d}x&=\frac{1}{\pi^{2}}\sum_{m\,\leq\,X^{2}}\frac{r^{2}_{{\scriptscriptstyle 2}}(m)}{m^{2}}+O_{\epsilon}\big(X^{-1+\epsilon}\big)\\
&=\frac{1}{\pi^{2}}\sum_{m=1}^{\infty}\frac{r^{2}_{{\scriptscriptstyle 2}}(m)}{m^{2}}+O_{\epsilon}\big(X^{-1+\epsilon}\big)\,.
\end{split}
\end{equation}
Recalling that $\Xi(m,1)=0$, it follows that the RHS of \eqref{eq:5.1} in the case where $j=2$ is given by
\begin{equation}\label{eq:5.9}
\begin{split}
\sum_{m=1}^{\infty}\Xi(m,2)&=\frac{1}{2\pi^{2}}\sum_{m=1}^{\infty}\frac{\mu^{2}(m)}{m^{2}}\sum_{\textit{e}_{1},\textit{e}_{2}=\pm1}\underset{\textit{e}_{1}k_{1}+\textit{e}_{2}k_{2}=0}{\sum_{k_{1}, k_{2}=1}^{\infty}}\prod_{i=1}^{2}\frac{r_{{\scriptscriptstyle2}}\big(mk^{2}_{i}\big)}{k_{i}^{2}}\\
&=\frac{1}{\pi^{2}}\sum_{m,k=1}^{\infty}\mu^{2}(m)\frac{r^{2}_{{\scriptscriptstyle2}}\big(mk^{2}\big)}{(mk^{2})^{2}}=\frac{1}{\pi^{2}}\sum_{m=1}^{\infty}\frac{r^{2}_{{\scriptscriptstyle 2}}(m)}{m^{2}}\,.
\end{split}
\end{equation}
This proves \eqref{eq:5.1} in the case where $j=2$.\\
\textbf{Case 3.} $j\geq3$. In what follows, all implied constants in the Big $O$ notation are allowed to depend on $j$. By Proposition $1$ and Lemma $4$, together with the trivial bound $|x^{-2}T\big(x^{2}\big)|\ll1$, we have
\begin{equation}\label{eq:5.10}
\frac{1}{X}\int\limits_{ X}^{2X}\widehat{\mathcal{E}}^{j}_{1}(x)\textit{d}x=\bigg(-\frac{\sqrt{2}}{\pi}\bigg)^{j}\frac{1}{X}\int\limits_{ X}^{2X}\bigg(\,\sum_{m\,\leq\,X^{2}}\frac{r_{{\scriptscriptstyle 2}}(m)}{m}\cos{\big(2\pi\sqrt{m}x\big)}\bigg)^{j}\textit{d}x+O_{\epsilon}\big(X^{-1+\epsilon}\big)\,.
\end{equation}
Let $1\leq Y\leq X^{2}$ be a large parameter to be determined later. In the range $X<x<2X$, we have
\begin{equation}\label{eq:5.11}
\begin{split}
&\Bigg(\,\sum_{m\,\leq\,X^{2}}\frac{r_{{\scriptscriptstyle 2}}(m)}{m}\cos{\big(2\pi\sqrt{m}x\big)}\Bigg)^{j}\\
&=\Bigg(\,\sum_{m\,\leq\,Y}\frac{r_{{\scriptscriptstyle 2}}(m)}{m}\cos{\big(2\pi\sqrt{m}x\big)}\Bigg)^{j}+O\bigg((\log{X})^{j-1}\bigg|\sum_{Y\,<\,m\,\leq\,X^{2}}\frac{r_{{\scriptscriptstyle 2}}(m)}{m}\exp{\big(2\pi i\sqrt{m}x\big)}\bigg|\,\bigg)\,.
\end{split}
\end{equation}
We first estimate the mean-square of the second summand appearing on the RHS of \eqref{eq:5.11}. We have
\begin{equation}\label{eq:5.12}
\begin{split}
&\frac{1}{X}\int\limits_{ X}^{2X}\bigg|\sum_{Y\,<\,m\,\leq\,X^{2}}\frac{r_{{\scriptscriptstyle 2}}(m)}{m}\exp{\big(2\pi i\sqrt{m}x\big)}\bigg|^{2}\textit{d}x\\
&=\sum_{Y\,<\,m\,\leq\,X^{2}}\frac{r^{2}_{{\scriptscriptstyle 2}}(m)}{m^{2}}+\frac{1}{2\pi iX}\sum_{Y\,<\,m\neq n\,\leq\,X^{2}}\frac{r_{{\scriptscriptstyle 2}}(m)}{m}\frac{r_{{\scriptscriptstyle 2}}(n)}{n}\frac{\exp{\big(2\pi i(\sqrt{m}-\sqrt{n}x)\big)}}{\sqrt{m}-\sqrt{n}}\bigg|^{x=2X}_{x=X}\,.
\end{split}
\end{equation}
Applying Hilbert's inequality, we find that
\begin{equation}\label{eq:5.13}
\begin{split}
\Bigg|\,\sum_{Y\,<\,m\neq n\,\leq\,X^{2}}\frac{r_{{\scriptscriptstyle 2}}(m)}{m}\frac{r_{{\scriptscriptstyle 2}}(n)}{n}\frac{\exp{\big(2\pi i(\sqrt{m}-\sqrt{n}x)\big)}}{\sqrt{m}-\sqrt{n}}\bigg|^{x=2X}_{x=X}\,\Bigg|&\ll\sum_{m\,>Y}r^{2}_{{\scriptscriptstyle2}}(m)m^{-3/2}\\
&\ll Y^{-1/2}\log{Y}\,.
\end{split}
\end{equation}
Inserting \eqref{eq:5.13} into the RHS of \eqref{eq:5.12}, we obtain
\begin{equation}\label{eq:5.14}
\begin{split}
\frac{1}{X}\int\limits_{ X}^{2X}\bigg|\sum_{Y\,<\,m\,\leq\,X^{2}}\frac{r_{{\scriptscriptstyle 2}}(m)}{m}\exp{\big(2\pi i\sqrt{m}x\big)}\bigg|^{2}\textit{d}x&=\sum_{Y\,<\,m\,\leq\,X^{2}}\frac{r^{2}_{{\scriptscriptstyle 2}}(m)}{m^{2}}+O\big(X^{-1}Y^{-1/2}\log{Y}\big)\\
&\ll Y^{-1}\log{Y}+X^{-1}Y^{-1/2}\log{Y}\\
&\ll Y^{-1}\log{Y}\,.
\end{split}
\end{equation}
Integrating both sides of \eqref{eq:5.11}, we have by \eqref{eq:5.10}, \eqref{eq:5.14} and Cauchy–Schwarz inequality
\begin{equation}\label{eq:5.15}
\begin{split}
\frac{1}{X}\int\limits_{ X}^{2X}\widehat{\mathcal{E}}^{j}_{1}(x)\textit{d}x&=\bigg(-\frac{\sqrt{2}}{\pi}\bigg)^{j}\frac{1}{X}\int\limits_{ X}^{2X}\bigg(\,\sum_{m\,\leq\,Y}\frac{r_{{\scriptscriptstyle 2}}(m)}{m}\cos{\big(2\pi\sqrt{m}x\big)}\bigg)^{j}\textit{d}x+O_{\epsilon}\big(X^{\epsilon}Y^{-1/2}\big)\\
&=\bigg(\frac{-1}{\sqrt{2}\pi}\bigg)^{j}\sum_{\textit{e}_{1},\ldots,\textit{e}_{j}=\pm1}\sum_{m_{1},\ldots,m_{j}\leq Y}\prod_{i=1}^{j}\frac{r_{{\scriptscriptstyle 2}}(m_{i})}{m_{i}}\frac{1}{X}\int\limits_{ X}^{2X}\cos{\bigg(2\pi\Big(\sum_{i=1}^{j}\textit{e}_{i}\sqrt{m_{i}}\Big)x\bigg)}\textit{d}x\\
&\,\,\,\,\,\,+O_{\epsilon}\big(X^{\epsilon}Y^{-1/2}\big)\,.
\end{split}
\end{equation}
Before we proceed to evaluate the RHS of \eqref{eq:5.15}, we need the following result (see \cite{Wenguang}, $\S2$ Lemma $2.2$). Let $\textit{e}_{1},\ldots,\textit{e}_{j}=\pm1$, and suppose that $m_{1},\ldots,m_{j}\leq Y$ are integers. Then it holds 
\begin{equation}\label{eq:5.16}
\sum_{i=1}^{j}\textit{e}_{i}\sqrt{m_{i}}\neq0\Longrightarrow\bigg|\sum_{i=1}^{j}\textit{e}_{i}\sqrt{m_{i}}\bigg|\gg Y^{1/2-2^{j-2}}\,,
\end{equation}
where the implied constant depends only $j$. It follows from \eqref{eq:5.16} that
\begin{equation}\label{eq:5.17}
\frac{1}{X}\int\limits_{ X}^{2X}\cos{\bigg(2\pi\Big(\,\sum_{i=1}^{j}\textit{e}_{i}\sqrt{m_{i}}\Big)x\bigg)}\textit{d}x=\left\{
        \begin{array}{ll}
            1 & ;\,\sum_{i=1}^{j}\textit{e}_{i}\sqrt{m_{i}}=0 
            \\\\
            O\big(X^{-1}Y^{2^{j-2}-1/2}\big) & ;\, \sum_{i=1}^{j}\textit{e}_{i}\sqrt{m_{i}}\neq0 \,.
        \end{array}
    \right.
\end{equation}
The estimate \eqref{eq:5.17} in the case where $\sum_{i=1}^{j}\textit{e}_{i}\sqrt{m_{i}}\neq0$ is somewhat wasteful, but nevertheless it will suffice for us. Inserting \eqref{eq:5.17} into the RHS of \eqref{eq:5.15}, we obtain
\begin{equation}\label{eq:5.18}
\begin{split}
\frac{1}{X}\int\limits_{ X}^{2X}\widehat{\mathcal{E}}^{j}_{1}(x)\textit{d}x&=\bigg(\frac{-1}{\sqrt{2}\pi}\bigg)^{j}\sum_{\textit{e}_{1},\ldots,\textit{e}_{j}=\pm1}\,\,\underset{\sum_{i=1}^{j}\textit{e}_{i}\sqrt{m_{i}}=0}{\sum_{m_{1},\ldots,m_{j}\,\leq\,Y}}\prod_{i=1}^{j}\frac{r_{{\scriptscriptstyle 2}}(m_{i})}{m_{i}}+O_{\epsilon}\Big(X^{\epsilon}Y^{-1/2}\Big\{1+X^{-1}Y^{2^{j-2}}\Big\}\Big)\,.
\end{split}
\end{equation}
It remains to estimate the first summand appearing on the RHS of \eqref{eq:5.18}.  Let $\textit{e}_{1},\ldots,\textit{e}_{j}=\pm1$. We have 
\begin{equation}\label{eq:5.19}
\begin{split}
\underset{\sum_{i=1}^{j}\textit{e}_{i}\sqrt{m_{i}}=0}{\sum_{m_{1},\ldots,m_{j}\,\leq\,Y}}\prod_{i=1}^{j}\frac{r_{{\scriptscriptstyle 2}}(m_{i})}{m_{i}}=\sum_{m_{1},\ldots,m_{j}\,\leq\,Y}\prod_{i=1}^{j}\frac{\mu^{2}(m_{i})}{m_{i}}\,\,\underset{\sum_{i=1}^{j}\textit{e}_{i}k_{i}\sqrt{m_{i}}=0}{\sum_{k_{1}\,\leq\,\sqrt{Y/m_{1}},\ldots,k_{j}\,\leq\,\sqrt{Y/m_{j}}}}\,\,\prod_{i=1}^{j}\frac{r_{{\scriptscriptstyle 2}}\big(m_{i}k^{2}_{i}\big)}{k^{2}_{i}}\,.
\end{split}
\end{equation}
Now, since the elements of the set $\mathscr{B}=\{\sqrt{m}:|\mu(m)|=1\}$ are are linearly independent over $\mathbb{Q}$, it follows that the the relation $\sum_{i=1}^{j}\textit{e}_{i}k_{i}\sqrt{m_{i}}=0$ with $m_{1},\ldots,m_{j}$ square-free, is equivalent to the relation $\sum_{i\in S_{\ell}}\textit{e}_{i}k_{i}=0$ for $1\leq\ell\leq s$, where $\biguplus_{\ell=1}^{s} S_{\ell}=\{1,,\ldots,j\}$ is a partition. Multiplying \eqref{eq:5.19} by $\big(-1/\sqrt{2}\pi\big)^{j}$, summing over all $\textit{e}_{1},\ldots,\textit{e}_{j}=\pm1$ and rearranging the terms in ascending order, it follows that
\begin{equation}\label{eq:5.20}
\begin{split}
\bigg(\frac{-1}{\sqrt{2}\pi}\bigg)^{j}\sum_{\textit{e}_{1},\ldots,\textit{e}_{j}=\pm1}&\,\,\underset{\sum_{i=1}^{j}\textit{e}_{i}\sqrt{m_{i}}=0}{\sum_{m_{1},\ldots,m_{j}\,\leq\,Y}}\prod_{i=1}^{j}\frac{r_{{\scriptscriptstyle 2}}(m_{i})}{m_{i}}\\
&=\sum_{s=1}^{j}\underset{\,\,\ell_{1},\,\ldots\,,\ell_{s}\geq1}{\sum_{\ell_{1}+\cdots+\ell_{s}=j}}\,\,\frac{j!}{\ell_{1}!\cdots\ell_{s}!}\sum_{m_{1}<\cdots<m_{s}\leq Y}\prod_{i=1}^{s}\Xi\big(m_{i},\ell_{i};\sqrt{Y/m_{i}}\,\big)\,.
\end{split}
\end{equation}
where for $y\geq1$, and integers $m,\ell\geq1$, the term $\Xi(m,\ell;y)$ is given
\begin{equation}\label{eq:5.21}
\Xi(m,\ell;y)=(-1)^{\ell}\bigg(\frac{1}{\sqrt{2}\pi}\bigg)^{\ell}\frac{\mu^{2}(m)}{m^{\ell}}\sum_{\textit{e}_{1},\ldots,\textit{e}_{\ell}=\pm1}\underset{\textit{e}_{1}k_{1}+\cdots+\textit{e}_{\ell}k_{\ell}=0}{\sum_{k_{1},\ldots,k_{\ell}\leq y}}\prod_{i=1}^{\ell}\frac{r_{{\scriptscriptstyle2}}\big(mk^{2}_{i}\big)}{k_{i}^{2}}\,.
\end{equation}
Denoting by $\tau(\cdot)$ the divisor function, we have for $m$ square-free
\begin{equation}\label{eq:5.22}
\sum_{k\,>y}\frac{r_{{\scriptscriptstyle2}}\big(mk^{2}\big)}{k^{2}}\ll r_{{\scriptscriptstyle2}}(m)\sum_{k\,>y}\frac{\tau^{2}(k)}{k^{2}}\ll r_{{\scriptscriptstyle2}}(m)y^{-1}\big(\log{2y}\big)^{3}\,.
\end{equation}
Using \eqref{eq:5.22} repeatedly, it follows that
\begin{equation}\label{eq:5.23}
\Xi(m,\ell;y)=\Xi(m,\ell)+O\bigg(\frac{\mu^{2}(m)r^{\ell}_{{\scriptscriptstyle2}}(m)}{m^{\ell}}y^{-1}\big(\log{2y}\big)^{3}\bigg)\,,
\end{equation}
where $\Xi(m,\ell)$ is given as in \eqref{eq:1.4}. Using \eqref{eq:5.23} repeatedly, and noting that $\Xi(m,1;y)=0$, we have by \eqref{eq:5.20}
\begin{equation}\label{eq:5.24}
\begin{split}
&\bigg(\frac{-1}{\sqrt{2}\pi}\bigg)^{j}\sum_{\textit{e}_{1},\ldots,\textit{e}_{j}=\pm1}\,\,\underset{\sum_{i=1}^{j}\textit{e}_{i}\sqrt{m_{i}}=0}{\sum_{m_{1},\ldots,m_{j}\,\leq\,Y}}\prod_{i=1}^{j}\frac{r_{{\scriptscriptstyle 2}}(m_{i})}{m_{i}}\\
&=\sum_{s=1}^{j}\underset{\,\,\ell_{1},\,\ldots\,,\ell_{s}\geq1}{\sum_{\ell_{1}+\cdots+\ell_{s}=j}}\,\,\frac{j!}{\ell_{1}!\cdots\ell_{s}!}\sum_{m_{1}<\cdots<m_{s}\leq Y}\prod_{i=1}^{s}\Xi(m_{i},\ell_{i})+O\big(Y^{-1/2}(\log{Y})^{3}\big)\\
&=\sum_{s=1}^{j}\underset{\,\,\ell_{1},\,\ldots\,,\ell_{s}\geq1}{\sum_{\ell_{1}+\cdots+\ell_{s}=j}}\,\,\frac{j!}{\ell_{1}!\cdots\ell_{s}!}\underset{\,\,\,m_{s}>\cdots>m_{1}}{\sum_{m_{1},\,\ldots\,,m_{s}=1}^{\infty}}\prod_{i=1}^{s}\Xi(m_{i},\ell_{i})+O\big(Y^{-1/2}(\log{Y})^{3}\big)\,,
\end{split}
\end{equation}
where the series appearing on the RHS of \eqref{eq:5.24} converges absolutely. Finally, inserting \eqref{eq:5.24} into the RHS of \eqref{eq:5.18} and making the choice $Y=X^{2^{2-j}}$, we arrive at
\begin{equation}\label{eq:5.25}
\begin{split}
\frac{1}{X}\int\limits_{ X}^{2X}\widehat{\mathcal{E}}^{j}_{1}(x)\textit{d}x&=\sum_{s=1}^{j}\underset{\,\,\ell_{1},\,\ldots\,,\ell_{s}\geq1}{\sum_{\ell_{1}+\cdots+\ell_{s}=j}}\,\,\frac{j!}{\ell_{1}!\cdots\ell_{s}!}\underset{\,\,\,m_{s}>\cdots>m_{1}}{\sum_{m_{1},\,\ldots\,,m_{s}=1}^{\infty}}\prod_{i=1}^{s}\Xi(m_{i},\ell_{i})+O_{\epsilon}\Big(X^{-\eta_{j}+\epsilon}\Big)\,,
\end{split}
\end{equation}
where $\eta_{j}=2^{1-j}$. This settles the proof in the case where $j\geq3$. The proof of Proposition $4$ is therefore complete. 
\end{proof}
\section{Proof of the main results: Theorem $1$ \& $2$}
\noindent
Collecting the results from the previous sections, we are now in a position to present the proof of the main results. We begin with the proof of Theorem 1.
%
\begin{proof}(Theorem $1$).
We shall first prove \eqref{eq:1.1} in the particular case where $\mathcal{F}\in C^{\infty}_{0}(\mathbb{R})$, that is, $\mathcal{F}$ is an infinitely differentiable function having compact support.\\
To that end, let $\mathcal{F}$ be test function as above, and note that the assumptions on $\mathcal{F}$ imply that $|\mathcal{F}(w)-\mathcal{F}(y)|\leq c_{\mathcal{F}}|w-y|$ for all $w,y$, where $c_{\mathcal{F}}>0$ is some constant which depends on $\mathcal{F}$. It follows that for any integer $M\geq1$ and any $X>0$, we have
\begin{equation}\label{eq:6.1}
\frac{1}{X}\int\limits_{ X}^{2X}\mathcal{F}\big(\widehat{\mathcal{E}}_{1}(x)\big)\textit{d}x=\frac{1}{X}\int\limits_{ X}^{2X}\mathcal{F}\Big(\sum_{m\leq M}\upphi_{{\scriptscriptstyle 1,m}}\big(\sqrt{m}x\big)\Big)\textit{d}x+\mathscr{E}_{\mathcal{F}}\big(X,M\big)\,,
\end{equation}
where $\upphi_{{\scriptscriptstyle 1,m}}(t)$ is defined as in Proposition $2$, and the remainder term $\mathscr{E}_{\mathcal{F}}\big(X,M\big)$ satisfies the bound
\begin{equation}\label{eq:6.2}
\big|\mathscr{E}_{\mathcal{F}}\big(X,M\big)\big|\leq c_{\mathcal{F}}\,\frac{1}{X}\int\limits_{X}^{2X}\Big|\widehat{\mathcal{E}}_{1}(x)-\sum_{m\leq M}\upphi_{{\scriptscriptstyle 1,m}}\big(\sqrt{m}x\big)\Big|\textit{d}x\,. 
\end{equation}
By Proposition $2$ it follows that
\begin{equation}\label{eq:6.3}
\lim_{M\to\infty}\limsup_{X\to\infty}\big|\mathscr{E}_{\mathcal{F}}\big(X,M\big)\big|=0\,.
\end{equation}
Denoting by $\widehat{\mathcal{F}}$ the Fourier transform of $\mathcal{F}$, the assumptions on the test function $\mathcal{F}$ allows us to write 
\begin{equation}\label{eq:6.4}
\frac{1}{X}\int\limits_{ X}^{2X}\mathcal{F}\Big(\sum_{m\,\leq\,M}\upphi_{{\scriptscriptstyle 1,m}}\big(\sqrt{m}x\big)\Big)\textit{d}x=\int\limits_{-\infty}^{\infty}\widehat{\mathcal{F}}(\alpha)\mathfrak{M}_{X}(\alpha;M)\textit{d}\alpha\,,
\end{equation}
where $\mathfrak{M}_{X}(\alpha;M)$ is defined at the beginning of $\S4$. Letting $X\to\infty$ in \eqref{eq:6.4}, we have by \eqref{eq:4.2} in Proposition $3$ together with an application of Lebesgue’s dominated convergence theorem
\begin{equation}\label{eq:6.5}
\lim_{X\to\infty}\frac{1}{X}\int\limits_{ X}^{2X}\mathcal{F}\Big(\sum_{m\,\leq\,M}\upphi_{{\scriptscriptstyle 1,m}}\big(\sqrt{m}x\big)\Big)\textit{d}x=\int\limits_{-\infty}^{\infty}\widehat{\mathcal{F}}(\alpha)\mathfrak{M}(\alpha;M)\textit{d}\alpha\,,
\end{equation}
where $\mathfrak{M}(\alpha;M)=\prod_{m\,\leq\,M}\Phi_{{\scriptscriptstyle1,m}}(\alpha)$, and $\Phi_{{\scriptscriptstyle1,m}}(\alpha)$ is defined at the beginning of $\S4$. Letting $M\to\infty$ in \eqref{eq:6.5}, and recalling that $\mathcal{P}_{1}(\alpha)=\widehat{\Phi}_{{\scriptscriptstyle1}}(\alpha)$ where $\Phi_{{\scriptscriptstyle1}}(\alpha)=\prod_{m=1}^{\infty}\Phi_{{\scriptscriptstyle1,m}}(\alpha)$, we have by Lebesgue’s dominated convergence theorem
%
\begin{equation}\label{eq:6.6}
\begin{split}
\lim_{M\to\infty}\lim_{X\to\infty}\frac{1}{X}\int\limits_{ X}^{2X}\mathcal{F}\Big(\sum_{m\,\leq\,M}\upphi_{{\scriptscriptstyle 1,m}}\big(\sqrt{m}x\big)\Big)\textit{d}x&=\int\limits_{-\infty}^{\infty}\widehat{\mathcal{F}}(\alpha)\Phi_{{\scriptscriptstyle1}}(\alpha)\textit{d}\alpha\\
&=\int\limits_{-\infty}^{\infty}\mathcal{F}(\alpha)\mathcal{P}_{1}(\alpha)\textit{d}\alpha\,,
\end{split}
\end{equation}
where in the second equality we made use of Parseval's theorem which is justified by the decay estimate \eqref{eq:4.4} for $\mathcal{P}^{(j)}_{1}(\alpha)$ with $\alpha$ real stated in Proposition $3$. It follows from \eqref{eq:6.4}, \eqref{eq:6.5} and \eqref{eq:6.6}  that
\begin{equation}\label{eq:6.7}
\begin{split}
\frac{1}{X}\int\limits_{ X}^{2X}\mathcal{F}\Big(\sum_{m\,\leq\,M}\upphi_{{\scriptscriptstyle 1,m}}\big(\sqrt{m}x\big)\Big)\textit{d}x=\int\limits_{-\infty}^{\infty}\mathcal{F}(\alpha)\mathcal{P}_{1}(\alpha)\textit{d}\alpha+\mathscr{E}^{\flat}_{\mathcal{F}}\big(X,M\big)\,,
\end{split}
\end{equation}
where the remainder term $\mathscr{E}^{\flat}_{\mathcal{F}}\big(X,M\big)$ satisfies
\begin{equation}\label{eq:6.8}
\lim_{M\to\infty}\lim_{X\to\infty}\big|\mathscr{E}^{\flat}_{\mathcal{F}}\big(X,M\big)\big|=0\,.
\end{equation}
Inserting \eqref{eq:6.7} into the RHS of \eqref{eq:6.1}, we deduce from \eqref{eq:6.3} and \eqref{eq:6.8} that
\begin{equation}\label{eq:6.9}
\begin{split}
\limsup_{X\to\infty}\bigg|\frac{1}{X}\int\limits_{ X}^{2X}\mathcal{F}\big(\widehat{\mathcal{E}}_{1}(x)\big)\textit{d}x&-\int\limits_{-\infty}^{\infty}\mathcal{F}(\alpha)\mathcal{P}_{1}(\alpha)\textit{d}\alpha\bigg|\\
&\leq\lim_{M\to\infty}\lim_{X\to\infty}\big|\mathscr{E}^{\flat}_{\mathcal{F}}\big(X,M\big)\big|+\lim_{M\to\infty}\limsup_{X\to\infty}\big|\mathscr{E}_{\mathcal{F}}\big(X,M\big)\big|=0\,.
\end{split}
\end{equation}
We conclude that
\begin{equation}\label{eq:6.10}
\lim\limits_{X\to\infty}\frac{1}{X}\int\limits_{X}^{2X}\mathcal{F}\big(\widehat{\mathcal{E}}_{1}(x)\big)\textit{d}x=\int\limits_{-\infty}^{\infty}\mathcal{F}(\alpha)\mathcal{P}_{1}(\alpha)\textit{d}\alpha\,,
\end{equation}
whenever $\mathcal{F}\in C^{\infty}_{0}(\mathbb{R})$.\\
The result \eqref{eq:6.10} extends easily to include the class $C_{0}(\mathbb{R})$ of continuous functions with compact support. To see this, fix a smooth bump function $\varphi(y)\geq0$ supported in $[-1,1]$ having total mass $1$, and for an integer $n\geq1$ let $\varphi_{n}(y)=n\varphi(ny)$. Given $\mathcal{F}\in C_{0}(\mathbb{R})$, let $\mathcal{F}_{n}=\mathcal{F}\star\varphi_{n}\in C^{\infty}_{0}(\mathbb{R})$, where $\star$ denotes the Euclidean convolution operator. We then have
\begin{equation}\label{eq:6.11}
\begin{split}
\bigg|\frac{1}{X}\int\limits_{X}^{2X}\mathcal{F}\big(&\widehat{\mathcal{E}}_{1}(x)\big)\textit{d}x-\int\limits_{-\infty}^{\infty}\mathcal{F}(\alpha)\mathcal{P}_{1}(\alpha)\textit{d}\alpha\bigg|\\
&\leq\bigg|\frac{1}{X}\int\limits_{X}^{2X}\mathcal{F}_{n}\big(\widehat{\mathcal{E}}_{1}(x)\big)\textit{d}x-\int\limits_{-\infty}^{\infty}\mathcal{F}(\alpha)\mathcal{P}_{1}(\alpha)\textit{d}\alpha\bigg|+\underset{y\in\mathbb{R}}{\textit{max}}\,\,\big|\mathcal{F}(y)-\mathcal{F}_{n}(y)\big|\,.
\end{split}
\end{equation}
Since $\underset{y\in\mathbb{R}}{\textit{max}}\,\,\big|\mathcal{F}(y)-\mathcal{F}_{n}(y)\big|\to0$ as $n\to\infty$, and
\begin{equation}\label{eq:6.12}
\begin{split}
\int\limits_{-\infty}^{\infty}\mathcal{F}(\alpha)\mathcal{P}_{1}(\alpha)\textit{d}\alpha&=\lim_{n\to\infty}\int\limits_{-\infty}^{\infty}\mathcal{F}_{n}(\alpha)\mathcal{P}_{1}(\alpha)\textit{d}\alpha\\
&=\lim_{n\to\infty}\lim\limits_{X\to\infty}\frac{1}{X}\int\limits_{X}^{2X}\mathcal{F}_{n}\big(\widehat{\mathcal{E}}_{1}(x)\big)\textit{d}x\,,
\end{split}
\end{equation}
it follows that \eqref{eq:6.10} holds whenever $\mathcal{F}\in C_{0}(\mathbb{R})$.\\
Let us now consider the general case in which $\mathcal{F}$ is a continuous function of polynomial growth, say $|\mathcal{F}(\alpha)|\ll|\alpha|^{j}$ for all sufficiently large $|\alpha|$, where $j\geq1$ is some integer. Let $\psi\in C^{\infty}_{0}(\mathbb{R})$ satisfy $0\leq\psi(y)\leq1$, $\psi(y)=1$ for $|y|\leq1$, and set $\psi_{n}(y)=\psi(y/n)$. Define $\mathcal{F}_{n}(y)=\mathcal{F}(y)\psi_{n}(y)\in C_{0}(\mathbb{R})$. For $n$ sufficiently large we have by Proposition $4$
\begin{equation}\label{eq:6.13}
\begin{split}
\frac{1}{X}\int\limits_{X}^{2X}\mathcal{F}\big(\widehat{\mathcal{E}}_{1}(x)\big)\textit{d}x&=\frac{1}{X}\int\limits_{X}^{2X}\mathcal{F}_{n}\big(\widehat{\mathcal{E}}_{1}(x)\big)\textit{d}x+O\Bigg(\frac{1}{n^{j}}\frac{1}{X}\int\limits_{ X}^{2X}\widehat{\mathcal{E}}^{2j}_{1}(x)\textit{d}x\Bigg)\\
&=\frac{1}{X}\int\limits_{X}^{2X}\mathcal{F}_{n}\big(\widehat{\mathcal{E}}_{1}(x)\big)\textit{d}x+O\bigg(\frac{1}{n^{j}}\bigg)\,,
\end{split}
\end{equation}
where the implied constant depends only on $j$ and the implicit constant appearing in the relation $|\mathcal{F}(\alpha)|\ll|\alpha|^{j}$. Since $\mathcal{F}_{n}\to\mathcal{F}$ pointwise as $n\to\infty$, it follows from the rapid decay of $\mathcal{P}_{1}(\alpha)$ that
\begin{equation}\label{eq:6.14}
\begin{split}
\int\limits_{-\infty}^{\infty}\mathcal{F}(\alpha)\mathcal{P}_{1}(\alpha)\textit{d}\alpha&=\lim_{n\to\infty}\int\limits_{-\infty}^{\infty}\mathcal{F}_{n}(\alpha)\mathcal{P}_{}(\alpha)\textit{d}\alpha\\
&=\lim_{n\to\infty}\lim\limits_{X\to\infty}\frac{1}{X}\int\limits_{X}^{2X}\mathcal{F}_{n}\big(\widehat{\mathcal{E}}_{1}(x)\big)\textit{d}x\,.
\end{split}
\end{equation}
We conclude from \eqref{eq:6.13} and \eqref{eq:6.14} that $\frac{1}{X}\int_{X}^{2X}\mathcal{F}\big(\widehat{\mathcal{E}}_{1}(x)\big)\textit{d}x\to\int_{-\infty}^{\infty}\mathcal{F}(\alpha)\mathcal{P}_{1}(\alpha)\textit{d}\alpha$ as $X\to\infty$. It follows that \eqref{eq:6.10} holds for all continuous functions of polynomial growth. The extension to include the class of (piecewise)-continuous functions of polynomial growth is now straightforward, and so \eqref{eq:1.1} is proved.\\
The decay estimates \eqref{eq:1.2} stated in Theorem $1$ have already been proved in Proposition $3$, and it remains to show that $\mathcal{P}_{1}(\alpha)$ defines a probability density. To that end, we note that the LHS of \eqref{eq:6.10} is real and non-negative whenever $\mathcal{F}$ is. Since $\mathcal{P}_{1}(\alpha)$ is continuous, by choosing a suitable test function $\mathcal{F}$ in \eqref{eq:6.10}, we conclude that $\mathcal{P}_{q}(\alpha)\geq0$ for real $\alpha$. Taking $\mathcal{F}\equiv1$ in \eqref{eq:6.10} we find $\int_{-\infty}^{\infty}\mathcal{P}_{1}(\alpha)\textit{d}\alpha=1$. The proof of Theorem 1 is therefor complete.  
\end{proof}
\noindent
The proof of Theorem $2$ is now immediate.
\begin{proof}(Theorem 2).
In the particular case where $\mathcal{F}(\alpha)=\alpha^{j}$ with $j\geq1$ an integer, we have by Theorem $1$
\begin{equation}\label{eq:6.15}
\lim\limits_{X\to\infty}\frac{1}{X}\int\limits_{ X}^{2X}\widehat{\mathcal{E}}^{j}_{1}(x)\textit{d}x=\int\limits_{-\infty}^{\infty}\alpha^{j}\mathcal{P}_{1}(\alpha)\textit{d}\alpha\,.
\end{equation}
It immediately follows from \eqref{eq:5.1} in Proposition $4$ that
\begin{equation}\label{eq:6.16}
\int\limits_{-\infty}^{\infty}\alpha^{j}\mathcal{P}_{1}(\alpha)\textit{d}\alpha=\sum_{s=1}^{j}\underset{\,\,\ell_{1},\,\ldots\,,\ell_{s}\geq1}{\sum_{\ell_{1}+\cdots+\ell_{s}=j}}\,\,\frac{j!}{\ell_{1}!\cdots\ell_{s}!}\underset{\,\,\,m_{s}>\cdots>m_{1}}{\sum_{m_{1},\,\ldots\,,m_{s}=1}^{\infty}}\prod_{i=1}^{s}\Xi(m_{i},\ell_{i})\,,
\end{equation}
where the series on the RHS of \eqref{eq:6.16} converges absolutely, and for integers $m,\ell\geq1$, the term $\Xi(m,\ell)$ is given as in \eqref{eq:1.4}.\\
By definition $\Xi(m,1)=0$, and so it follows from \eqref{eq:6.16} that $\int_{-\infty}^{\infty}\alpha\mathcal{P}_{1}(\alpha)\textit{d}\alpha=0$. Suppose now that $j\geq3$ is an integer which satisfies $j\equiv1\,(2)$, and consider a summand
\begin{equation}\label{eq:6.17}
\underset{\,\,\,m_{s}>\cdots>m_{1}}{\sum_{m_{1},\,\ldots\,,m_{s}=1}^{\infty}}\prod_{i=1}^{s}\Xi(m_{i},\ell_{i})\,,
\end{equation}
with $\ell_{1}+\cdots+\ell_{s}=j$. We may clearly assume that $\ell_{1},\ldots,\ell_{s}\geq2$, for otherwise \eqref{eq:6.17} vanishes. As $j$ is odd, it follows that $|\{1\leq i\leq s: \ell_{i}\equiv1\,(2)\}|$ is also odd. Since $\Xi(m,\ell)<0$ for $\ell>1$ odd, and $\Xi(m,\ell)>0$ for $\ell$ even, it follows that \eqref{eq:6.17} is strictly negative. We conclude that $\int_{-\infty}^{\infty}\alpha^{j}\mathcal{P}_{1}(\alpha)\textit{d}\alpha<0$. The proof of Theorem $2$ is therefore complete. 
\end{proof}
\noindent
\textbf{Acknowledgements.} I would like to express my sincere gratitude to Prof. Amos Nevo for his support and encouragement throughout the writing of this paper. I would also like to thank Prof. Ze{\'e}v Rudnick for a stimulating discussion on this fascinating subject, which subsequently lead to the writing of this paper.

\end{document}